%% file: UnramifiedASv5.tex
\documentclass[12pt,reqno]{amsart}
\usepackage{amssymb,amscd,amsmath}
\usepackage[alphabetic]{amsrefs}
\usepackage[all]{xy}
\usepackage[headings]{fullpage}
\usepackage[T1]{fontenc}
\usepackage{libertine}
\usepackage{stackengine}
\usepackage{mathtools}
\usepackage{graphicx}
\usepackage{multirow}
\usepackage{colortbl}
\usepackage{pdfsync}
\usepackage{tikz}

\include{formatting}
\include{macros}
\numberwithin{equation}{section}

\newcommand{\D}{\mathbb{D}}
\newcommand{\Dk}{\mathbb{D}_k}
\newcommand{\et}{{\acute et}}
\newcommand{\Zps}{\Z/p\Z}
\renewcommand{\and}{\quad\text{and}\quad}
\newcommand{\Hy}{\mathbb{H}}
\renewcommand{\t}{{}^t}

\DeclareMathOperator{\Gal}{Gal}

\begin{document}
\title{$\MakeLowercase{p}$-torsion for unramified
  %$\Z/p\Z$
  Artin--Schreier covers of curves}

\author{Bryden Cais}
\address{Department of Mathematics \\ University of Arizona
 \\ Tucson, AZ~~85721 USA}
\email{cais@arizona.edu}

\author{Douglas Ulmer}
\address{Max Planck Institute for Mathematics, Vivatsgasse 7,
  53111 Bonn Germany}
\address{Department of Mathematics \\ University of Arizona
 \\ Tucson, AZ~~85721 USA}
\email{ulmer@arizona.edu}

\date{\today}

% this is really the 2020 classification, but \subjclass is broken and
% lists 1991 when given option 2020 or no option

% descriptions:
% 11G20 (1980-now) Curves over finite and local fields
% 14F40 (1980-now) de Rham cohomology and algebraic geometry
% 14H40 (1973-now) Jacobians, Prym varieties

% 11G10 (1980-now) Abelian varieties of dimension >1
% 14G17 (2010-now) Positive characteristic ground fields in algebraic geometry
% 14H30 (1973-now) Coverings of curves, fundamental group


\begin{abstract}
  Let $Y\to X$ be an unramified Galois cover of curves over a perfect
  field $k$ of characteristic $p>0$ with $\Gal(Y/X)\cong\Z/p\Z$, and
  let $J_X$ and $J_Y$ be the Jacobians of $X$ and $Y$ respectively.
  We consider the $p$-torsion subgroup schemes $J_X[p]$ and
  $J_Y[p]$, analyze the Galois-module structure
  of $J_Y[p]$, and find restrictions this structure imposes on $J_Y[p]$ (for
  example, as manifested in its Ekedahl--Oort type) taking $J_X[p]$
  as given.
\end{abstract}

\maketitle

\section{Introduction}
Let $k$ be a perfect field of characteristic $p>0$ with algebraic
closure $\kbar$, and let $X$ be a smooth, proper, geometrically
irreducible curve of genus $g_X$ over $k$.  Let $\pi:Y\to X$ be an
unramified Galois covering with $G:=\gal(Y/X)\cong\Z/p\Z$ and with
$Y$ geometrically irreducible.  Writing $g_Y$ for the genus of $Y$,
the Riemann-Hurwitz formula for $\pi$ says $2g_Y-2=p(2g_X-2)$.

Let $J_X$ and $J_Y$ be the Jacobians of $X$ and $Y$, and let $J_X[p]$
and $J_Y[p]$ be their $p$-torsion subgroup schemes. These are
self-dual $BT_1$ group schemes of orders $p^{2g_X}$ and $p^{2g_Y}$
respectively. (See Section~\ref{s:Dieu} for the definitions of
``self-dual'' and ``$BT_1$''.) Our goal is to describe relations
between $J_Y[p]$ and $J_X[p]$. We first consider the $G$-module
structure of $J_Y[p]$ and show that the latter is close to being
$G$-free in a suitable sense. We then explore the restrictions this
structure imposes on $J_Y[p]$ (for example, as manifested in its
Ekedahl--Oort type) taking $J_X[p]$ as given. In particular, over a
general perfect base field $k$, we deduce inequalities for the
``arithmetic $p$-rank'' of $Y$, namely $\dim_{\Fp}J_Y[p](k)$, and over
an algebraically closed base field $k$, we deduce new restrictions on
the isomorphism type of $J_Y[p]$ which in turn yield new parity
restrictions on the $a$-number of $Y$.

Recall that a $p$-torsion group scheme $G$ over $k$ has a canonical
decomposition
\[\GG\cong \GG_{\et}\oplus \GG_{m}\oplus \GG_{ll}\]
into \'etale, multiplicative, and local-local parts.  (See
Section~\ref{s:Dieu} below.)  Let $f_X$ and $f_Y$ be the $p$-ranks of
$X$ and $Y$ respectively (i.e., the dimensions over $\Fp$ of
$J_X[p]_{\et}(\kbar)$ and $J_Y[p]_{\et}(\kbar)$).  The
Deuring--Shafarevich formula in this context says that
$f_Y-1=p(f_X-1)$.  There is a well-known refinement of this taking
into account the $G$ action.  Indeed, let $\Fp[G]$ be the group ring
of $G$ over $\Fp$.  If $k$ is algebraically closed, then we have
isomorphisms of group schemes with $G$ action:
  \begin{align}
  J_Y[p]_{\et}&\cong \Z/p\Z\oplus
                \left((\Z/p\Z)^{f_X-1}\tensor_{\Fp}\Fp[G]\right)\label{eq:DSet}\\
                  \noalign{and}
  J_Y[p]_{m}&\cong\mu_p\oplus
                 \left((\mu_p)^{f_X-1}\tensor_{\Fp}\Fp[G]\right)\label{eq:DSm}
  \end{align}                  
  where $G$ acts trivially on the factors $\Z/p\Z$ and $\mu_p$.  (This
  follows from \cite[Theorem~2]{Nakajima85} or
  \cite[Theorem~1.5]{Crew84} and Cartier duality.  It may also be
  obtained from the Hochschild-Serre spectral sequences for $\pi$ with
  coefficients in $\Z/p\Z$ and $\mu_p$.)

\subsection*{Filtrations from the $G$ action and their associated
  gradeds}
Our first aim in this paper is to study all of $J_Y[p]$ (i.e., the
local-local part as well as the \'etale and multiplicative parts) and
to allow $k$ to be a general perfect field of characteristic $p$
(giving ``structural'' results analogous to \eqref{eq:DSet} and
\eqref{eq:DSm} for general $k$).  To state the result, we first define
a certain subquotient $\GG_X$ of $J_X[p]$.

  \begin{def-prop}\label{def:GG_X}
    The covering $\pi:Y\to X$ and the isomorphism
    $\gal(Y/X)\cong\Z/p\Z$ give rise to canonical homomorphisms of
    groups schemes $J_X[p]\onto\Z/p\Z$ and $\mu_P\into J_X[p]$.
    Define a \textup{(}self dual $BT_1$\textup{)} group scheme $\GG_X$
    by the exact sequences
  \begin{equation}\label{eq:GG_X,et}
0\to\GG_{X,\et}\to J_X[p]_{\et}\to\Z/p\Z\to0,    
  \end{equation}
  \begin{equation}\label{eq:GG_X,m} 
0\to\mu_p \to J_X[p]_m\to \GG_{X,m}\to0,    
  \end{equation}
and the isomorphism
\[\GG_{X,ll}\cong J_X[p]_{ll}.\]
  \end{def-prop}

  Let $\gamma\in G=\gal(Y/X)$ be the element corresponding to
  $1\in\Z/p\Z$ under the fixed isomorphism $G\cong\Z/p\Z$ and let
  $\delta$ be the element $\gamma-1$ in the group ring $k[G]$.  Then
  $k[G]\cong k[\delta]/(\delta^p)$.  Thus $\delta$ induces a nilpotent
  endomorphism of $J_Y[p]$, and the kernels and images of powers of
  $\delta$ on $J_Y[p]$ give two (generally distinct) $p$-step
  filtrations by self-dual $BT_1$ group schemes.
  The following result describes the associated graded objects of
  these two filtrations. (We use the standard notations $\GG[\delta]$
  and $\GG/\delta$ to denote kernels and cokernels respectively.)

\begin{thm}\label{thm:G-str-genl-k}\mbox{}
   \begin{enumerate}
  \item   There are canonical isomorphisms
    \[ \frac{\delta^{i}J_Y[p]_{\et}}{\delta^{i+1}J_Y[p]_{\et}}\cong
      \frac{J_Y[p]_{\et}[\delta^{p+1-i}]}{J_Y[p]_{\et}[\delta^{p-i}]}
      \cong \GG_{X,\et}
      \quad\text{for $i=1,\dots,p-1$},\]
 as well as exact sequences
    \begin{equation}\label{eq:GG_Y,et,ker} 
    0\to\GG_{X,\et}\to J_Y[p]_{\et}[\delta]\to\Z/p\Z\to0
  \end{equation}
  and
      \begin{equation}\label{eq:GG_Y,et,coker}
    0\to\Z/p\Z\to J_Y[p]_{\et}/\delta\to\GG_{X,\et}\to0.
  \end{equation}
Pull back by $\pi$ induces  a canonical isomorphism
$J_X[p]_\et\cong J_Y[p]_{\et}[\delta]$ which identifies the exact
sequences \eqref{eq:GG_Y,et,ker} and \eqref{eq:GG_X,et}. 
  \item   There are canonical isomorphisms
    \[  \frac{\delta^{i}J_Y[p]_{m}}{\delta^{i+1}J_Y[p]_{m}}\cong
      \frac{J_Y[p]_{m}[\delta^{p+1-i}]}{J_Y[p]_{m}[\delta^{p-i}]}
       \cong \GG_{X,m}
      \quad\text{for $i=1,\dots,p-1$},\]
     as well as exact sequences
    \begin{equation}\label{eq:GG_Y,m,coker}
    0\to\mu_p\to J_Y[p]_{m}/\delta\to\GG_{X,m}\to0
  \end{equation}
  and
      \begin{equation}\label{eq:GG_Y,m,ker}
    0\to\GG_{X,m}\to J_Y[p]_{m}[\delta]\to\mu_p\to0.
  \end{equation}
  Push forward by $\pi$ induces a canonical isomorphism
  $J_Y[p]_{m}/\delta\cong J_X[p]_m$ which identifies the exact
  sequences \eqref{eq:GG_Y,m,coker} and \eqref{eq:GG_X,m}. 
\item We have equalities
  \[\delta^{i}J_Y[p]_{ll}=J_Y[p]_{ll}[\delta^{p-i}]\]
  for $i=0,\dots,p$, as well as canonical isomorphisms
  \[\frac{\delta^{i}J_Y[p]_{ll}}{\delta^{i+1}J_Y[p]_{ll}}\cong
    \frac{J_Y[p]_{ll}[\delta^{p-i}]}{J_Y[p]_{ll}[\delta^{p-i-1}]}
\cong    J_X[p]_{ll}
        \quad\text{for $i=0,\dots,p-1$.}\]
\end{enumerate}
\end{thm}

As we will see in Section~\ref{s:H1dR}, the asymmetry between the
kernels and cokernels of $\delta$ (i.e., \eqref{eq:GG_Y,et,ker} vs
\eqref{eq:GG_Y,et,coker} and \eqref{eq:GG_Y,m,coker} vs
\eqref{eq:GG_Y,m,ker}) is significant.  Indeed, although
$J_Y[p]_{\et}[\delta]$ and $J_Y[p]_{\et}/\delta$ have the same order
and are closely related, they are not in general isomorphic.
Similarly for $J_Y[p]_{m}[\delta]$ and $J_Y[p]_{m}/\delta$.
Readers are referred to Figures~\ref{fig:1} and \ref{fig:2} in
Section~\ref{s:H1dR} for pictorial versions of
Theorems~\ref{thm:G-str-genl-k} and \ref{thm:k-point} in terms of
Dieudonn\'e modules.  Among other things, the figures show how the two
filtrations (by images and kernels of $\delta$) interact.

When $k=\kbar$, we may recover the isomorphisms~\eqref{eq:DSet} and
\eqref{eq:DSm} from parts (1) and (2) of the theorem using the fact
that the category of $p$-torsion \'etale (resp. multiplicative) group
schemes over $k$ is semi-simple with unique simple object $\Z/p\Z$
(resp.~$\mu_p$). The situation for the local-local part is much more
complicated even when $k$ is algebraically closed and will be discussed
in more detail below.

We now consider a certain freeness property of $p$-torsion group
schemes with $G$ action.

  \begin{def-lemma}\label{def:free}
    Let $\GG$ be a finite commutative group scheme over $k$ killed by
    $p$ and equipped with an action of $G=\Z/p\Z$, i.e., a group
    scheme equipped with the structure of a module over $\Fp[G]$.  We
    say $\GG$ is \emph{$G$-free} if the following equivalent
    conditions are satisfied:
    \begin{enumerate}
    \item the Dieudonn\'e module $M(\GG)$ is free over the group ring
      $k[G]$
    \item $\GG[\delta]/\delta^{p-1}\GG=0$
    \item $\GG[\delta^{p-1}]/\delta\GG=0$
    \item $\delta^{p-1}$ induces an isomorphism $\GG/\delta\isoto\GG[\delta]$
    \end{enumerate}
  \end{def-lemma}

\begin{cor}\label{cor:ll-free}
  $J_Y[p]_{ll}$ is $G$-free.
\end{cor}

\begin{proof}
  Conditions (2), (3), and (4) in Definition~\ref{def:free} follow
  immediately from part (3) of Theorem~\ref{thm:G-str-genl-k}.  The
  corollary also follows from Theorem~\ref{thm:k-point} just below.
\end{proof}

There is an elegant, uniform variant of Theorem~\ref{thm:G-str-genl-k}
provided that $X$ has a $k$-rational point.

\begin{thm}\label{thm:k-point}
  Suppose that $X$ has a $k$-rational point $S$, and let
  $T=\pi^{-1}(S)$ viewed as a closed subscheme of $Y$.  Then there is
  a self-dual $BT_1$ group scheme $\HH$ equipped with the structure of
  a module over $\Fp[G]$ with the following properties:
  \begin{enumerate}
  \item $\HH$ is $G$-free in the sense of
    Definition-Lemma~\ref{def:free}.
\item There are equalities $\delta^{i}\HH=\HH[\delta^{p-i}]$ for
  $i=1,\dots,p$, as well as canonical isomorphisms
  \[      \frac{\delta^{i}\HH}{\delta^{i+1}\HH}\cong
    \frac{\HH[\delta^{p-i}]}{\HH[\delta^{p-i-1}]}\cong
    J_X[p]
        \quad\text{for $i=0,\dots,p-1$.}\]
    \item There are canonical exact sequences
      \[0\to J_Y[p]_{\et}\to\HH_{\et}\to \res_{T/S}\Z/p\Z\to\Z/p\Z\to0,\]
      and
      \[0\to \mu_p\to\res_{T/S}\mu_p\to\HH_m\to J_Y[p]_m\to0,\]
        and a canonical isomorphism
        \[\HH_{ll}\cong J_Y[p]_{ll}.\]
    \end{enumerate}
\end{thm}

The group scheme $\HH$ depends on the choice of $S$ in an
interesting way, see Section~\ref{ss:S-dependence}.

Theorems~\ref{thm:G-str-genl-k} and \ref{thm:k-point} identify the
minimal subquotients of $J_Y[p]$ and $\HH$ as $\Fp[G]$-modules, and
one might hope to ``reassemble'' the group schemes from this
information. However, the category of $BT_1$ group schemes is not well
behaved with respect to extensions (even when $k$ is algebraically
closed), so even taking $J_X[p]$ as known, the structure of the
repeated extensions $J_Y[p]$ and $\HH$ can be quite intricate. For
example, in Section~\ref{s:comments} we show that there can be no
Jordan-H\"older theorem for $BT_1$ group schemes, and we show that
there may be infinitely many non-isomorphic $\D_k[G]$-modules giving
rise to a fixed isomorphism class of $\D_k$-module and a fixed
isomorphism class of $k[G]$-module. Nevertheless, we are able to
deduce significant restrictions on $J_Y[p]$.

\smallskip
\subsection*{Analysis of the \'etale part of $J_Y[p]$}
We now consider freeness and related splitting questions for the
\'etale part of $J_Y[p]$.  Similar results hold for the multiplicative
part by Cartier duality, and we leave it to the reader to make them
explicit.

Note that equation~\eqref{eq:DSet} implies that when $k$ is
algebraically closed, $J_Y[p]_\et$ is the direct sum of $\Z/p\Z$ and a
$G$-free group scheme.  The following result gives criteria for
the same structural result to hold over a general $k$.
  
\begin{prop}\label{prop:splitting}\mbox{}
  \begin{enumerate}
    \item The exact sequence \eqref{eq:GG_Y,et,ker} splits if and only
      if there is an exact sequence of $k$-group schemes
      \[0\to\Z/p\Z\to J_Y[p]_\et\to\QQ\to 0\]
      where $\QQ$ is $G$-free.
    \item The exact sequence \eqref{eq:GG_Y,et,coker} splits if and only
      if there is an exact sequence of $k$-group schemes
      \[0\to\KK\to J_Y[p]_\et\to\Z/p\Z\to 0\]
      where $\KK$ is $G$-free.      
    \item The exact sequences \eqref{eq:GG_Y,et,ker} and
       \eqref{eq:GG_Y,et,coker} both split if and only if $J_Y[p]_\et$
       is the direct sum of $\Z/p\Z$ and a $G$-free group scheme.
   \end{enumerate}
\end{prop}

We will see in Example~\ref{ex:splitting} that \eqref{eq:GG_Y,et,ker}
and \eqref{eq:GG_Y,et,coker} may or may not split, and splitting of
one does not in general imply splitting of the other; similarly for
\eqref{eq:GG_Y,m,coker} and \eqref{eq:GG_Y,m,ker}.

For a commutative $p$-torsion group scheme $\GG$ over $k$, define the
\emph{arithmetic $p$-rank of $\GG$}, denoted $\nu(G)$, by
\[p^{\nu(\GG)}=|\GG(k)|=|\GG_\et(k)|.\] Let $\nu_X=\nu(J_X[p])$ and
$\nu_Y=\nu(J_Y[p])$.

We say that $J_X[p]_\et$ (resp. $J_Y[p]_\et$) is \emph{completely
  split} over $k$ if $\nu_X=f_X$ (resp. $\nu_Y=f_Y$), or equivalently,
if $J_X[p]_\et\cong\left(\Z/p\Z\right)^{f_X}$ (resp. if
$J_Y[p]_\et\cong\left(\Z/p\Z\right)^{f_Y}$ as group schemes ignoring
the $G$ action).  It follows from Theorem~\ref{thm:G-str-genl-k} that
$J_Y[p]_\et$ is completely split over $k$ if and only if we have an
isomorphism of group schemes with $G$ action as in
equation~\eqref{eq:DSet}.

We have the following general results on the
arithmetic $p$-ranks of $X$ and $Y$.

\begin{thm}\label{thm:et-inequalities}
  \mbox{}
  \begin{enumerate}
  \item   We have $\nu_X\le \nu_Y\le p\nu_X$.\label{thm:et-inequalities:part1}
  \item If the exact sequence \eqref{eq:GG_X,et} $($equivalently
    \eqref{eq:GG_Y,et,ker}$)$ splits, then we have the stronger upper
    bound $\nu_Y-1\le p(\nu_X-1)$.\label{thm:et-inequalities:part2}
 \item If $k$ is finite or algebraically closed, then
   $\nu_X\ge1$.\label{thm:et-inequalities:part3} 
  \item If $k$ is finite, \eqref{eq:GG_Y,et,ker} is split, and
    \eqref{eq:GG_Y,et,coker} is non-split, then
    $\nu_X\ge 2$.\label{thm:et-inequalities:part4}
  \end{enumerate}
  \end{thm}

See Example~\ref{ex:nu_X=0} for a discussion of the surprising possibility
that $\nu_X$ might be 0 if $k$ is infinite.

We also have bounds on the degrees of extensions of $k$ over which
certain splitting behaviors occur:

\begin{thm}\label{thm:et-extensions}
  \mbox{}
  \begin{enumerate}
  \item There is a finite Galois extension $k'$ of $k$ of exponent
    dividing $p$ such that the sequences \eqref{eq:GG_X,et} and
    \eqref{eq:GG_Y,et,ker} split over $k'$.  Thus \textup{(}by
    Proposition~\ref{prop:splitting}\,\textup{)}, over $k'$, $J_Y[p]_\et$ is an
    extension of $G$-free group scheme by $\Z/p\Z$.
  \item If $J_X[p]_\et$ is completely split then there is a finite
    Galois extension $k'$ over $k$ with exponent dividing $p$ such
    that $J_Y[p]_\et$ is completely split over $k'$.  
  \end{enumerate}
  \end{thm}

  In particular, if $k$ is finite, the splitting behavior in the Theorem
  happens over extensions $k'/k$ of degree dividing $p$.

  A splitting not discussed in the theorem (going from sequences
  \eqref{eq:GG_X,et} and \eqref{eq:GG_Y,et,ker} being split to
  $J_X[p]_\et$ being completely split) is controlled by
  $\GL_{f_X-1}(\Fp)$, and unfortunately, the exponent of this group
  grows rapidly with $f_X$.  
  
We can give a very complete description of the group scheme
$J_Y[p]_\et$ when $k$ is finite and $f_X\le2$.

\begin{thm}\label{thm:f_X=2}
  \mbox{}
  \begin{enumerate}
    \item[(A)] If $f_X=1$, then $J_Y[p]_\et\cong J_X[p]_\et\cong\Z/p\Z$.
\item[(B)]  Suppose that $p>2$, $k$ is finite, and $f_X=2$.  Then exactly
  one of the following holds:
\begin{enumerate}
\item[(1a)] $\nu_X=\nu_Y=1$.  In this case, there is an isomorphism
  \begin{equation}\label{eq:alpha<>1}
  J_Y[p]_\et\cong \Z/p\Z  \oplus \GG   
  \end{equation}
  where $\GG$ is $G$-free of rank 1 and $\nu(\GG)=0$. There is a
  nontrivial extension $k'$ of $k$ of degree dividing $p-1$ such that
  $|\GG(k')|=p^\mu$ where $1\le \mu\le p$, and there is an extension
  $k''$ of $k$ of degree dividing $p$ such that $\GG\cong(\GG')^p$
  over $k''$ where $\GG'$ has rank 1 and $\nu(\GG')=0$.  Over $k'k''$,
  $J_Y[p]_\et$ is completely split.
\item[(1b)] $\nu_X=1$ and $\nu_Y>1$.  In this case, $\nu_Y<p+1$
  \textup{(}so $J_Y[p]_\et$ is not completely split over
  $k$\textup{)}, and $J_Y[p]_\et$ is completely split over the
  extension of $k$ of degree $p$.
\item[(2)] $\nu_X=2$.  In this case, $2\le\nu_Y\le p+1$ and
  $J_Y[p]_\et$ is completely split over an extension of degree
  dividing $p$.
\end{enumerate}
\item[(C)]  Suppose that $p=2$, $k$ is finite, and $f_X=2$.  Then exactly
  one of the following holds:
\begin{enumerate}
\item[(1)] $\nu_X=1$.  In this case, $1\le\nu_Y<3$ \textup{(}so
  $J_Y[p]_\et$ is not completely split over $k$\textup{)}, and
  $J_Y[p]_\et$ splits completely over an extension of $k$ of degree
  dividing 4.
\item[(2)] $\nu_X=2$.  In this case, $2\le\nu_Y\le3$, there is an exact
  sequence
  \[0\to\Z/p\Z\to J_Y[p]_\et\to\QQ\to0\]
  where $\nu(\QQ)\ge1$, and $J_Y[p]_\et$
  splits completely over an extension of $k$ of degree dividing $p$.
\end{enumerate}

  \end{enumerate}
\end{thm}

See the proof in Section~\ref{s:apps-etale} for  additional information on the
structure of $J_Y[p]_\et$ in each case.

\subsection*{Analysis of the local-local part of $J_Y[p]$}
For the local-local part, it seems hopeless to give a full analysis,
even when $k$ is algebraically closed.
Nevertheless, certain cases can be described rather explicitly. Recall
that the $a$-number of a finite group scheme $\JJ$ over $k$ is the
largest integer $a$ such that there is an injection
$\alpha_p^a\into\JJ$. (When $\JJ$ is local-local, the $a$-number also
has an interpretation as the number of generators and relations of the
Dieudonn\'e module of $\JJ$.) Write $a_X$ and $a_Y$ for the
$a$-numbers of $J_X[p]$ and $J_Y[p]$ respectively. Booher and Cais
\cite[\S6E]{BooherCais20} have observed that in our context,
$a_X\le a_Y\le pa_X$. We can improve and refine this in some cases.

We say that a $BT_1$ group scheme is ``superspecial'' if it is
isomorphic to a power of $E_{ss}[p]$ where $E_{ss}$ is a supersingular
elliptic curve.

\begin{thm}\label{thm:h=1-app}
  Assume that $p>2$ and that $k$ is algebraically closed.
  \begin{enumerate}
  \item If $f_X=g_X-1$ \textup{(}which implies that $a_X=1$ and
    $J_X[p]_{ll}\cong E_{ss}[p]$\textup{)}, then
    \[a_Y\in\{2,4,\dots,p-1,p\}.\]
   Moreover the local-local part
    of $J_Y[p]$ has an explicit description in terms of
    generators and relations depending only on $a_Y$.
  \item More generally, if the local-local part of $J_X[p]$ is
   superspecial and $a_Y<p$, then $a_Y$ is even.
 \item If $J_X[p]_{ll}$ is superspecial and $a_Y=pa_X$ or $pa_X-1$
   then there are few possibilities for the local-local part of
   $J_Y[p]$, and they can all be described explicitly in terms of
   generators and relations. 
  \end{enumerate}
 \end{thm}

 Part (1) is proven as Theorem~\ref{thm:h=1} below, and the precise
 description in terms of generators and relations is given there using
 Dieudonn\'e modules.  Part (2) and other features of the case where
 $J_X[p]_{ll}$ is superspecial are given in
 Theorem~\ref{thm:more-a-parity}, and part (3) is proven as
 Theorem~\ref{thm:large-a}.

 When $p=2$, we can completely analyze the possibilities for
 $J_Y[p]_{ll}$ when $J_X[p]_{ll}$ is superspecial. They can be stated
 in terms of Ekedahl--Oort structures or multisets of words on the
 alphabet $\{f,v\}$. (See \cite{Oort01} or \cite{PriesUlmer21} for
 background.)

\begin{thm}\label{thm:J_X-superspecial}
  Suppose that $p=2$ and that $J_X[p]_{ll}$ is superspecial. Then the
  Ekedahl--Oort structure of $J_Y[p]_{ll}$ starts with $h$ zeroes,
  i.e., it has the form $[0,0,\dots,0,\psi_{h+1},\dots,\psi_{ph}]$.
The multiset of words associated to $J_Y[p]_{ll}$ consists of $fv$ with
even multiplicity and of words obtained by concatenating words of the
form $f^2(vf)^{e_1}v^2(fv)^{e_2}$ with $e_1,e_2\ge0$.
\end{thm}

The theorem reduces the number of possibilities for $J_X[p]_{ll}$ from
$2^{2h-1}$ to $2^{h}$.  It will be proven in Theorem~\ref{thm:p=2} via
a complete analysis of Dieudonn\'e modules with the properties enjoyed
by $J_Y[p]_{ll}$ in this context.

\subsection*{Examples} 
Throughout the paper, we report on numerous examples and counterexamples
computed in {\sc Magma} using code extending \cite{Weir}.

\subsection*{Outline of the paper}
We now describe the main outlines of the paper.  By a theorem of Oda,
the Dieudonn\'e module of $J_Y[p]$ is isomorphic to the de Rham
cohomology $H^1_{dR}(Y)$, so we are lead to study the action of $G$ on
this and related cohomology groups.  After discussing preliminaries on
Dieudonn\'e modules and $k[G]$-modules in Sections~\ref{s:Dieu} and
\ref{s:G-mods}, we will prove crucial results of Chevalley-Weil type
on the $G$-module structure of various flat and coherent cohomology
groups in Section~\ref{s:H1dR} and deduce the Dieudonn\'e module
version of Theorems~\ref{thm:G-str-genl-k} and \ref{thm:k-point}.  In
particular, we show that $H^1_{dR}(Y)$ is close to being a free module
over the group ring $k[G]$, and we control the Dieudonn\'e structure
of the ``errors''.  See Propositions~\ref{prop:H^1(Y)} and
\ref{prop:H^1(Y,T)} for the precise statements.  The translation from
modules to groups is given in Section~\ref{s:groups}.

In Section~\ref{s:comments}, we review some of the difficulties in
recovering $J_Y[p]$ from its associated graded as given in
Theorem~\ref{thm:G-str-genl-k}.

Our results on the \'etale part of $J_Y[p]$
(Proposition~\ref{prop:splitting} and
Theorems~\ref{thm:et-inequalities}, \ref{thm:et-extensions}, and
\ref{thm:f_X=2}) are proven in Section~\ref{s:apps-etale}, mostly
using the correspondence between \'etale group schemes and Galois
representations. In Section~\ref{s:apps-ll}, we axiomatize
properties of the Dieudonn\'e module of the local-local part of $J_Y[p]$,
introduce a useful set of coordinates, and then prove the results underlying
Theorems~\ref{thm:h=1-app}, and \ref{thm:J_X-superspecial}.

\bigskip

\subsection*{Standing notation} We fix the following notation and
hypotheses for the rest of the paper: $k$ is a perfect field of
characteristic $p>0$ with algebraic closure $\kbar$; $X$ is a smooth,
proper, geometrically irreducible curve of genus $g_X$ over $k$;
$\pi:Y\to X$ is an unramified Galois covering with group $\Z/p\Z$ such
that $Y$ has genus $g_Y$ and is geometrically irreducible; and we fix an
isomorphism $G:=\gal(Y/X)\cong\Z/p\Z$.

\subsection*{Acknowledgements} 
The authors thank Klaus Lux for clarifying some subtleties of modular
representation theory.  BC was partially supported by NSF Grants DMS-2302072 and DMS-1902005. DU was
partially supported by Simons collaboration grant 713699.  DU also
thanks the Max Planck Institute for Mathematics in Bonn for its
support during the latter stages of the preparation of this article.

\section{Dieudonn\'e theory}\label{s:Dieu}
We will use Dieudonn\'e theory to study $J_X[p]$ and
$J_Y[p]$.  We refer to \cite{Fontaine77} for the basic facts recalled below.

Let $\Dk$ be the associative $k$-algebra generated by symbols $F$ and
$V$ with relations
\[FV=VF=0,\qquad F\alpha=\alpha^pF,\and \alpha V=V\alpha^p\]
for all $\alpha\in k$.

A \emph{$p$-group scheme over $k$} is by definition a finite
commutative group scheme over $k$ which is annihilated by $p$.

Recall that Dieudonn\'e theory gives a contravariant equivalence
between the category of $p$-group schemes over $k$ and the category of
left $\Dk$-modules of finite length.  If $G$ is a $p$-group scheme
over $k$, we write $M(G)$ for the corresponding $\Dk$-module.

By definition, a $\Dk$-module $M$ is \emph{self-dual} if it admits a
non-degenerate pairing of $\Dk$-modules, i.e., a non-degenerate
$k$-bilinear pairing $\<\cdot,\cdot\>$ with the properties that
\begin{equation}\label{eq:D-module-pairing}
\<Fx,y\>=\<x,Vy\>^p\and\<Vx,y\>=\<x,Fy\>^{1/p}  
\end{equation}
for all $x,y\in M$.  By definition, a $\Dk$-module $M$ is a
\emph{$BT_1$ module} if $\Ker(F)=\im(V)$ (or equivalently, if
$\im(F)=\Ker(V)$).

By definition, a $p$-group scheme $G$ over $k$ is \emph{self-dual}
(resp.~a \emph{$BT_1$ group scheme}) if $M(G)$ is self-dual (resp.~a
$BT_1$ module).

Any finite-dimensional $\D_k$-module $N$ decomposes as
\begin{equation}\label{eq:et-m-ll}
N\cong N_{\et}\oplus N_m\oplus N_{ll}  
\end{equation}
where $N_{\et}$ is \'etale ($F$ is bijective and $V=0$), $N_m$ is
multiplicative ($F=0$ and $V$ is bijective), and $N_{ll}$ is
``local-local'' ($F$ and $V$ are nilpotent).  (Choose a sufficiently
large integer $a$ and set $N_{\et}=\im F^a$, $N_m=\im V^a$, and
$N_{ll}=\Ker F^a\cap \Ker V^a$.)

The assignments $N\rightsquigarrow N_{\et}, N_m, N_{ll}$ are exact
functors on the category of $\D_k$-modules.  We denote the
corresponding functors on $p$-torsion group schemes by
$\GG\rightsquigarrow\GG_{\et},\GG_m$ and $\GG_{ll}$.

If $N$ is a $\D_k[G]$-module, the decomposition is also respected by
the $G$ action since $G$ commutes with $F$ and $V$.  Also, $N$ is
self-dual if and only if $N_{\et}$ is dual to $N_m$ and $N_{ll}$ is
self-dual; and $N$ is a $BT_1$ module if and only if $N_{ll}$ is a
$BT_1$ module.

A theorem of Oda \cite[Cor.~5.11]{Oda69} says that for a smooth,
proper, irreducible curve $Z$ over $k$ with Jacobian $J_Z$, the
$p$-torsion subgroup $J_Z[p]$ is a self-dual $BT_1$ group scheme, and
$M(J_Z[p])\cong H^1_{dR}(Z)$ where $H^1_{dR}(Z)$ is equipped with a
natural $\Dk$-module structure.\footnote{Oda's result requires that
  $Z$ have a $k$-rational point.  This is of course no restriction
  when $k$ is algebraically closed.  When $k$ is only assumed to be
  perfect, \cite[Prop.~5.4]{Cais10} shows that Oda's result continues
  to hold even without a rational point.}  We will use this result to
prove the Theorems~\ref{thm:G-str-genl-k} and \ref{thm:k-point} as
statements about $H^1_{dR}(X)$ and $H^1_{dR}(Y)$ and related
Dieudonn\'e modules.

\section{$G$-modules}\label{s:G-mods}
Consider the group rings $k[G]$ or $\Fp[G]$ where as before
$G=\gal(Y/X)$ and we have fixed an isomorphism $G\cong\Z/p\Z$.
Let $\gamma\in G$ be the element corresponding to $1\in\Z/p\Z$.  Then
\[k[G]\cong k[\gamma]/(\gamma^p-1)\cong k[\delta]/(\delta^p)\]
where $\delta:=\gamma-1$.  Note that
\[\delta^{p-1}=\gamma^{p-1}+\cdots+1\]
is the trace element of $k[G]$.

It is easily checked that up to isomorphism the indecomposable
$k[G]$-modules are
\[V_i:=k[\delta]/(\delta^i)\quad\text{for $i=1,\dots,p$.}\] (See,
e.g., \cite[Lemma~64.2]{CurtisReinerRTFGAA}).  By the Krull--Schmidt
theorem, every finitely generated $k[G]$-module is (non-canonically)
isomorphic to a direct sum of indecomposable modules.

\begin{lemma}\label{lemma:G-free}
Let $M$ be a finitely generated  $k[G]$-module.  Then the following
conditions are equivalent:
\begin{enumerate}
\item $M$ is free over $k[G]$,
\item $M[\delta]=\delta^{p-1}M$,
\item $M[\delta^{p-1}]=\delta M$,
\item $\delta^{p-1}$ induces an isomorphism $M/\delta M\cong M[\delta]$.
\end{enumerate}
\end{lemma}

\begin{proof}
  This follows immediately from the classification of indecomposable
  $k[G]$-modules and a straightforward calculation.
\end{proof}

Equivalence of the conditions in Definition-Lemma~\ref{def:free}
follows from Lemma~\ref{lemma:G-free} by applying the Dieudonn\'e
functor.

Define the dual of $V_i$ as $V_i^\vee=\Hom_k(V_i,k)$ with
action $(\gamma\phi)(v)=\phi(\gamma^{-1}v)$ for $\phi\in V_i^\vee$ and $v\in V_i$.
Then $V_i^\vee\cong V_i$ (non-canonically) as $k[G]$-modules.

A non-degenerate, bilinear form $\<\cdot,\cdot\>$ on $M$ such that
$\<\gamma m,\gamma n\>=\<m,n\>$ for all $m,n\in M$ induces an
isomorphism $M\cong M^\vee$ of $k[G]$-modules.

Defining $\tilde\delta:=\gamma^{-1}-1=-\gamma^{-1}\delta$,  we have
\begin{equation*}
  \<m,\delta n\>=\<m,\gamma n\>-\<m,n\>
  =\<\gamma^{-1}m,n\>-\<m,n\>=\<\tilde\delta m,n\>.  
\end{equation*}
for all $m,n\in M$.  Note as well that $\delta$ and $\tilde\delta$
have the same image and kernel on any $k[G]$-module.

Parallel definitions and results hold for $\Fp[G]$-modules.  We write
$W_j$ for the module $\Fp[\delta]/(\delta^j)$ over
$\Fp[G]\cong\Fp[\delta]/(\delta^p)$.

\section{de Rham cohomology as a $\Dk[G]$-module}\label{s:H1dR}
Readers are assumed to be familiar with the flat, coherent, and de
Rham cohomology of curves over perfect fields, and in particular with
the semi-linear endomorphisms $F$ and $V$ of the de Rham cohomology of
a curve.  We recommend \cite{MilneEC}, \cite{MumfordOdaAG2}, and
\cite{Oda69} as basic references.

Suppose that $Z$ is a smooth, proper, irreducible curve over $k$.
Then we have coherent cohomology groups $H^s(Z,\OO_Z)$ and
$H^s(Z,\Omega^1_Z)$, as well as de Rham cohomology groups
$H^s_{dR}(Z)$.  These are finite-dimensional $k$ vector spaces,
and there is an exact sequence
\begin{equation}\label{eq:dR-seq}
0\to H^0(Z,\Omega^1_Z)\to H^1_{dR}(Z)\to H^1(Z,\OO_Z)\to0.  
\end{equation}

There is a cup product on $H^1_{dR}(Z)$ which induces a perfect
alternating pairing
\[H^1_{dR}(Z)\times H^1_{dR}(Z)\to H^2_{dR}(Z)=k\]
denoted $\<\cdot,\cdot\>_Z$. The subspace
$H^0(Z,\Omega^1_Z)$ is isotropic, and the pairing restricts to the
(perfect) Serre duality pairing
\[H^0(Z,\Omega^1_Z)\times H^1(Z,\OO_Z)\to k.\]

There are also semi-linear operators $F$ and $V$ on $H^s_{dR}(Z)$ making
it into a $\Dk$-module.  Explicitly,
\[H^0_{dR}(Z)\cong k,\text{ with }F\alpha=\alpha^p\text{ and
  }V\alpha=0,\]
\[H^2_{dR}(Z)\cong k,\text{ with }F\alpha=0\text{ and
  }V\alpha=\alpha^{1/p},\]
in other words
\[H^0_{dR}(Z)\cong M(\Z/p\Z)\and H^2_{dR}(Z)\cong M(\mu_p).\] 
If $(\omega_i,f_{ij})$ is a hypercocycle for
an affine open cover $\{U_i\}$ of $Z$ representing a class
$c\in H^1_{dR}(Z)$, then $Fc$ and $Vc$ are represented by
\begin{equation}\label{eq:D-deR}
(0,f_{ij}^p)\and (\CC\omega_i,0)  
\end{equation}
respectively, where $\CC$ is the Cartier operator.  See \cite{Oda69}
for more details.

We have
\[\im\left(V:H^1_{dR}(Z)\to H^1_{dR}(Z)\right)
  =\Ker\left(F:H^1_{dR}(Z)\to H^1_{dR}(Z)\right)=H^0(Z,\Omega^1),\]
so $H^1_{dR}(Z)$ is a $BT_1$ module.  The pairing is compatible with
the $\Dk$-module structure in the sense of
equation~\eqref{eq:D-module-pairing}, so $H^1_{dR}(Z)$ is a self-dual
$BT_1$ module.

If $\pi$ is a finite surjective map of smooth, projective curves, we
have maps $\pi^*$ and $\pi_*$ on de Rham cohomology which are
compatible with the $\Dk$-module structures.  Also, if $\pi$ is a
Galois cover, $\pi^*\pi_*$ is the trace map.  Applied to $\pi:Y\to X$,
this means
\begin{equation}\label{eq:trace}
  \pi^*\pi_*=1+\gamma+\cdots+\gamma^{p-1}=\delta^{p-1}
\end{equation}
as endomorphisms of de Rham cohomology. 

With data $\pi:Y\to X$, $G=\gal(Y/X)\cong\Z/p\Z$ as usual, we will
prove two results on the cohomology of $X$ and $Y$
(Propositions~\ref{prop:H^1(Y)} and \ref{prop:H^1(Y,T)}) which will
yield Theorems~\ref{thm:G-str-genl-k} and \ref{thm:k-point}.

Recall that the $p$-rank of a curve $Z$ is by definition the integer $f_Z$
such that $J_Z[p](\kbar)\cong(\Z/p\Z)^{f_Z}$.  It is also equal to the
dimension over $k$ of $H^1_{dR}(Z)_\et$.

\begin{prop}\label{prop:H^1(Y)}\mbox{}
\begin{enumerate}
\item There are canonical homomorphisms of $\Dk$-modules
  \[M(\Z/p\Z)\into H^1_{dR}(X)_{\et}\into H^1_{dR}(X)\and
    H^1_{dR}(X)\onto H^1_{dR}(X)_m\onto M(\mu_p)\]
  which are exchanged by Cartier duality.  The Dieudonn\'e module of
  the group $\GG_X$ in Definition~\ref{def:GG_X} is
  \[\MM_X:=\frac{\Ker\left(H^1_{dR}(X)\onto M(\mu_p)\right)}
    {\im\left(M(\Z/p\Z)\into H^1_{dR}(X)\right)}.\]
\item There are \textup{(}non-canonical\textup{)} isomorphisms of
  $k[G]$-modules
  \begin{align*}
    H^1_{dR}(Y)_\et&\cong V_1\oplus V_p^{f_X-1},\\
    H^1_{dR}(Y)_m&\cong V_1\oplus V_p^{f_X-1},\\
    H^1_{dR}(Y)_{ll}&\cong V_p^{2h_X},\\
        \noalign{and}
    H^1_{dR}(Y)&\cong V_1^2\oplus V_p^{2g_X-2},
  \end{align*}
  where $h_X=g_X-f_X$.
\item $\pi^*$ induces isomorphisms of $\Dk$-modules
\[H^1_{dR}(X)_m\labeledto{\pi^*} H^1_{dR}(Y)_m[\delta]\and
  H^1_{dR}(X)_{ll}\labeledto{\pi^*} H^1_{dR}(Y)_{ll}[\delta],\]
and an exact sequence
\[0\to M(\Z/p\Z)\to H^1_{dR}(X)_\et\labeledto{\pi^*}
  H^1_{dR}(Y)_\et[\delta]\to M(\Z/p\Z)\to 0.\]
The image $\pi^*\left(H^1_{dR}(X)_\et\right)$ is equal to
  $\delta^{p-1}H^1_{dR}(Y)_\et$. 
\item $\pi_*$ induces an isomorphism
  \[H^1_{dR}(Y)_\et/\delta
    H^1_{dR}(Y)_\et\labeledto{\pi_*}H^1_{dR}(X)_\et\]
  which identifies the line
  \begin{multline*}
    \im\left(H^1_{dR}(Y)_\et[\delta]\to
      H^1_{dR}(Y)_\et/\delta H^1_{dR}(Y)_\et\right)\\
    =\Ker\left(H^1_{dR}(Y)_\et/\delta H^1_{dR}(Y)_\et
      \labeledlongto{\delta^{p-1}} H^1_{dR}(Y)_\et[\delta]\right)
  \end{multline*}
  with the line
  \[\im\left(M(\Z/pZ)\into H^1_{dR}(X)_\et\right)\]
  defined in part \textup{(1)}.
\end{enumerate}
\end{prop}

\begin{rems}\label{rem:dualities}\mbox{}
  \begin{enumerate}
  \item It is straightforward to check that the cup product on de Rham
cohomology induces a duality between $H^1_{dR}(Y)_\et$ and
$H^1_{dR}(Y)_m$ and its restriction to $H^1_{dR}(Y)_{ll}$ is perfect
and gives the latter the structure of a self-dual module.  It is also
easy to see that $H^1_{dR}(Y)_{ll}$ is a $BT_1$.  Since the pairing
satisfies the compatibility $\<\delta m,n\>_Y=\<m,\tilde\delta  n\>_Y$
(see Section~\ref{s:G-mods}), we find that the orthogonal complement
of $\delta^iH^1_{dR}(Y)_{ll}$ is $\delta^{p-i}H^1_{dR}(Y)_{ll}$.  This
implies that
\[\frac{H^1_{dR}(Y)_{ll}}{\delta^iH^1_{dR}(Y)_{ll}}
  \quad\text{is dual to}\quad
  H^1_{dR}(Y)_{ll}[\delta^i].\]
On the other hand, $\delta^{p-i}$ induces an isomorphism
\[\frac{H^1_{dR}(Y)_{ll}}{\delta^iH^1_{dR}(Y)_{ll}}\to
  H^1_{dR}(Y)_{ll}[\delta^i].\] Therefore, each of the submodules
$H^1_{dR}(Y)_{ll}[\delta^i]$ is self-dual, and they are easily seen to
be $BT_1$ modules as well.  (One slight subtlety: These pairings are
not compatible with restriction.  Indeed, the restriction to
$H^1_{dR}(Y)_{ll}[\delta^{i-1}]$ of the pairing just constructed on
$H^1_{dR}(Y)_{ll}[\delta^i]$ is degenerate for $i>1$.)
\item The maps in part (1) may also be interpreted as a 3-step
  filtration on $H^1_{dR}(X)$ with
  \begin{align*}
    H^1_{dR}(X)^3&=H^1_{dR}(X),\\
    H^1_{dR}(X)^2&=\Ker\left(H^1_{dR}(X)\to M(\mu_p)\right),\\
    H^1_{dR}(X)^1&=\im\left(M(\Z/p\Z)\to H^1_{dR}(X)\right),\\
    \noalign{\qquad\quad and}
    H^1_{dR}(X)^0&=0.
  \end{align*}
  The subquotients are $M(\mu_p)$, $\MM_X$,
  and $M(\Z/p\Z)$.
  \item The filtration above is self-dual in the sense that
    $H^1_{dR}(X)^2$ and $H^1_{dR}(X)^1$ are orthogonal complements
    of one another.
  \end{enumerate}
\end{rems}

Figure~\ref{fig:1} may help to digest the statement of the
Proposition.  It illustrates the case $p=5$, $g_X=5$, and $f_X=4$.
Each box represents a one-dimensional subspace of $H^1_{dR}(Y)$ (the
upper group) or $H^1_{dR}(X)$ (the lower group).  On $H^1_{dR}(Y)$,
the action of $\delta$ shifts a given one-dimensional subspace to the
one represented by the box below (if there is one, otherwise to zero).
The groups on the left represent the multiplicative parts, those on
the right represent the \'etale parts, and those in the middle
represent the local-local parts.  The (canonical) class
$\eta_X\in H^1_{dR}(X)_\et$ is constructed in the proof and spans the
image of the injection in part (1) (i.e., the subspace
$H^1_{dR}(X)^1$.) .  The (non-canonical) class
$\omega_X\in H^1_{dR}(X)_m$ maps onto a class spanning the image of
the projection in part (1) (i.e., the subquotient
$H^1_{dR}(X)/H^1_{dR}(X)^2$).  We may choose $\omega_X$ so that
$\<\omega_X,\eta_X\>=1$.  The (non-canonical) classes $\omega_Y$ and
$\eta_Y$ can be chosen to satisfy $\pi^*\omega_X=\omega_Y$,
$\pi_*\eta_Y=\eta_X$, and $\<\omega_Y,\eta_Y\>=1$.  Caution: In
general, the lines spanned by $\omega_X$, $\omega_Y$, and $\eta_Y$ are
not invariant under $\Dk$.  This is closely related to the possible
non-splitting of the exact sequences \eqref{eq:GG_Y,et,ker} and
\eqref{eq:GG_Y,et,coker}.

\begin{proof}[Proof of Proposition~\ref{prop:H^1(Y)}]
  (1) Since $\pi:Y\to X$ is unramified, the choice of a fixed isomorphism
$\gal(Y/X)\cong\Z/p\Z$ makes $Y$ into a $\Z/p\Z$-torsor over $X$.  The
group that classifies such torsors is $H^1_\et(X,\Zps)$ (see, for example,
\cite[III.4]{MilneEC} or \cite[6.5.5]{PoonenRPV}). Therefore, there is a class in
$H^1_\et(X,\Zps)$ defined by the cover $\pi:Y\to X$ and the fixed
isomorphism $G\cong\Z/p\Z$, which will be denoted by $\eta_{X,\et}$.

Consider the exact sequence
\[0\to\Z/p\Z\to\OO_X\labeledto{\wp}\OO_X\to0\]
of sheaves for the \'etale topology on $X$ where $\wp(x)=x^p-x$.
Taking cohomology yields a homomorphism
\[H^1_{\et}(X,\Z/p\Z)\to H^1_\et(X,\OO_X)[\wp],\]
where $[\wp]$ indicates the kernel of $\wp$.  Since $H^0(Y,\OO_Y)=k$,
the image of $\eta_{X,\et}$ in $H^1_\et(X,\OO_X)[\wp]$ is non-zero,
and since $H^1_\et(X,\OO_X)[\wp]$ is the subset of the usual
(coherent) $H^1(X,\OO_X)$ fixed by Frobenius, we have a class
$\eta_{X,coh}\in H^1(X,\OO_X)$ fixed by Frobenius.  It has a canonical
lift to $H^1_{dR}(X)$ (take any lift and apply $F$), and we denote
this lift by $\eta_X$.  Since $F\eta_X=\eta_X$, we see that
$\eta_X\in H^1_{dR}(X)_\et$, and the injection in (1) is
$\alpha\mapsto \alpha\eta_X$.

The surjection in (1) is obtained from the injection by Cartier
duality, and is given more explicitly by $c\mapsto\<c,\eta_X\>_X$.

Decomposing the module $\MM_X$ into its \'etale, multiplicative, and
local-local parts shows that
\[\MM_{X,\et}=\coker\left(M(\Z/p\Z)\to H^1_{dR}(X)_\et\right),\]
\[\MM_{X,m}=\Ker\left(H^1_{dR}(X)_m\to M(\mu_p)\right),\]
and
\[\MM_{X,ll}=H^1_{dR}(X)_{ll}.\]
Thus $\MM_X$ is the Dieudonn\'e module of $\GG_X$.  This establishes
part (1) of the Proposition.

\begin{figure}
\begin{center}
\begin{tikzpicture}[scale=0.7]
% Y, m
  \draw[step=1cm] (-2,0) grid (-5,5);
  \draw[step=1cm] (-6,2) grid (-7,3) node[anchor=north west] {$\omega_Y$};
% Y, ll
  \draw[step=1cm] (-1,0) grid (1,5);
%Y, et
  \draw[step=1cm] (2,0) grid (5,5);
  \draw[step=1cm] (6,2) grid (7,3) node[anchor=north east] {$\eta_Y$};
%X,m
  \draw[step=1cm] (-6,-3) grid (-7,-2) node[anchor=north west] {$\omega_X$};
  \draw[step=1cm] (-2,-3) grid (-5,-2) node[anchor=south west] {$\uparrow\pi^*$};
%X, ll
  \draw[step=1cm] (-1,-3) grid (1,-2);
%X, et
  \draw[step=1cm] (2,-3) grid (5,-2) node[anchor=south east] {$\downarrow\pi_*$};
  \draw[step=1cm] (6,-3) grid (7,-2) node[anchor=north east] {$\eta_X$};
\end{tikzpicture}
\end{center}
\caption{}\label{fig:1}
\end{figure}
\bigskip

(2) Taking the multiplicative and \'etale parts of the de Rham
sequence~\eqref{eq:dR-seq} with $Z=Y$ yields isomorphisms
\[H^1_{dR}(Y)_m\cong H^0(Y,\Omega^1_Y)_m\and
  H^1_{dR}(Y)_\et\cong H^1(Y,\OO_Y)_\et.\]

Tamagawa \cite{Tamagawa51} proves that there is an isomorphism of
$k[G]$-modules 
\begin{equation}\label{eq:Tamagawa}
H^0(Y,\Omega^1_Y)\cong V_1\oplus V_p^{g_X-1},  
\end{equation}
and Serre duality yields
\[H^1(Y,\OO_Y)\cong V_1\oplus V_p^{g_X-1}.\]
Nakajima \cite{Nakajima85} proves that there is an isomorphism of
$k[G]$-modules 
\begin{equation}\label{eq:Nakajima}
H^1_{dR}(Y)_m\cong H^0(Y,\Omega^1_Y)_m\cong V_1\oplus V_p^{f_X-1},  
\end{equation}
and Cartier duality yields
\[H^1_{dR}(Y)_\et\cong H^1(Y,\Omega^1_Y)_\et\cong V_1\oplus V_p^{f_X-1}.\]
This establishes the first two claims in part (2).

For the third claim in part (2), let $h_X=g_X-f_X$, and note that the
last four displayed equations and \eqref{eq:dR-seq} show that
$H^1_{dR}(Y)_{ll}$ is an extension of $V_p^{h_X}$ by $V_p^{h_X}$.
Since $V_p$ is free, the extension splits and there is an isomorphism
of $k[G]$-modules
\[H^1_{dR}(Y)_{ll}\cong V_p^{2h_X}.\]

The fourth claim in part (2) is simply the direct sum of the three
preceding claims.  This completes the proof of part (2) of the
Proposition.

(3) Consider the Hochschild--Serre spectral sequence for $\pi$
in de Rham cohomology.  The sequence of low degree terms is
\begin{equation*}
0\to H^1(G,H^0_{dR}(Y))\to
  H^1_{dR}(X)\labeledto{\pi^*}H^1_{dR}(Y)[\delta]\to
  H^2(G,H^0_{dR}(Y)).  
\end{equation*}
We have
\[H^1(G,H^0_{dR}(Y))\cong H^2(G,H^0_{dR}(Y))\cong M(\Z/p\Z),\]
and these modules are \'etale, so taking multiplicative
and local-local parts of the sequence yields isomorphisms
\[H^1_{dR}(X)_m\labeledto{\pi^*}H^1_{dR}(Y)_m[\delta]\and
  H^1_{dR}(X)_{ll}\labeledto{\pi^*}H^1_{dR}(Y)_{ll}[\delta].\]
Taking the \'etale part yields an exact sequence
\begin{equation}\label{eq:HS-dR}
  0\to M(\Z/p\Z)\to H^1_{dR}(X)_\et\labeledto{\pi^*}
  H^1_{dR}(Y)_\et[\delta]\to M(\Z/p\Z)\to 0,
\end{equation}
where surjectivity on the right follows from a dimension count using
part (2).  Thus $\eta_X$ spans the kernel of $\pi^*$ in
\eqref{eq:HS-dR}.  We will prove the last claim in (3) after establishing (4).

(4) The Cartier dual of the isomorphism
\[H^1_{dR}(X)_m\labeledto{\pi^*}H^1_{dR}(Y)_m[\delta]\]
is the isomorphism
\[H^1_{dR}(Y)_m/\delta H^1_{dR}(Y)_m\labeledto{\pi_*}H^1_{dR}(X)_m.\]
It follows from part (2) that the image of the natural map
\[H^1_{dR}(Y)_\et[\delta]\to H^1_{dR}(Y)_\et/\delta
    H^1_{dR}(Y)_\et\]
is a line, and that it is equal to
\[\Ker\left(H^1_{dR}(Y)_\et/\delta H^1_{dR}(Y)_\et
  \labeledto{\delta^{p-1}}
  H^1_{dR}(Y)_\et[\delta] \right),\]
and therefore equal to 
\[\Ker\left(H^1_{dR}(Y)_\et/\delta H^1_{dR}(Y)_\et
  \labeledto{\pi^*\pi_*}
  H^1_{dR}(Y)_\et[\delta]\right). \]
This shows that $\pi_*$ identifies this line with
\[\Ker \left(H^1_{dR}(X)_\et\labeledto{\pi^*} H^1_{dR}(Y)_\et[\delta]\right),\]
and we observed above that this is the line spanned by $\eta_X$, i.e.,
the line defined in part (1).  This completes the proof of part (4).
We also deduce that
\[\pi^*\left(H^1_{dR}(X)_\et\right)=\pi^*\pi_*
  \left(H^1_{dR}(Y)_\et\right)
  =\delta^{p-1}\left(H^1_{dR}(Y)_\et\right),\]
and this establishes the last claim in part (3).
\end{proof}

We now turn to an elegant result that holds under the assumption that
$X$ has a $k$-rational point.

\begin{prop}\label{prop:H^1(Y,T)}
Assume that $X$ has a $k$-rational point $S$ and let
$T=\pi^{-1}(S)$. Let $\NN_Y$ be the hypercohomology group
\[\NN_Y:=\Hy^1\left(Y,\OO_Y(-T)\labeledto{d}\Omega^1_Y(T)\right).\]
Then
\begin{enumerate}
\item $\NN_Y\cong V_p^{2g_X}$ as $k[G]$-modules.
\item For $i=1,\dots,p$, there are isomorphisms
  \[\frac{\NN_Y[\delta^i]}{\NN_Y[\delta^{i-1}]}
    =\frac{\delta^{p-i}\NN_Y}{\delta^{p-i+1}\NN_Y}
    \cong H^1_{dR}(X)\]
  of $\Dk$-modules.
\item There are exact sequences of $\D_k[G]$-modules
  \[0\to H^0(Y,\OO_Y)\to H^0(Y,\OO_T) \to \NN_{Y,\et}\to
    H^1_{dR}(Y)_\et\to0,\]
  and
  \[0\to H^1_{dR}(Y)_m\to\NN_{Y,m}\to
H^0(Y,\Omega^1_Y(T)/\Omega^1_Y)
\to H^1(Y,\Omega^1_Y)\to0,\]
as well as an isomorphism
\[\NN_{Y,ll}\cong H^1_{dR}(Y)_{ll}.\]
\end{enumerate}
\end{prop}

The $\Dk[G]$-module structures and homomorphisms in part (3) will be
made explicit in the proof.

\begin{rems}\label{rem:more-on-pairings2}\mbox{}
  \begin{enumerate}
\item In parallel with Remark~\ref{rem:dualities},  each of the subquotients
    $\NN_Y[\delta^i]/\NN_Y[\delta^j]$ for $0\le i<j\le p$, is
    self-dual, but not in a way compatible with restrictions of pairings.
\item The exact sequences appearing in part (3) may also be
  interpreted as a 3-step filtration on $\NN_Y$ with
  \begin{align*}
    \NN_Y^3&=\NN_Y,\\
    \NN_Y^2&=\ker\left(\NN_Y\to  H^0(Y,\Omega^1_Y(T)/\Omega^1_Y)\right),\\
    \NN_Y^1&=\im\left(H^0(Y,\OO_T)\to \NN_Y\right),\\
    \noalign{\qquad\quad and}
    \NN_Y^0&=0.
  \end{align*}
  The subquotients are
\begin{multline*}
  \qquad\qquad \Ker\left( H^0(Y,\Omega^1_Y(T)/\Omega^1_Y)
    \to H^1(Y,\Omega^1_Y)\right),\quad
H^1_{dR}(Y),\quad\text{and}\quad\\
\coker\left(H^0(Y,\OO_Y)\to H^0(Y,\OO_T)\right).\qquad\qquad
\end{multline*}
  \item The filtration $\NN_Y^i$ is self-dual in the sense that
    $\NN_Y^2$ and $\NN_Y^1$ are orthogonal complements to one another.
  \item The filtration on $\NN_Y$ induces a filtration on each of the
  subquotients $\NN_Y[\delta^i]/\NN_Y[\delta^{i+1}]$, and the induced
  filtration is the one on $H^1_{dR}(X)$ in Remark~\ref{rem:dualities}.
  \end{enumerate}    
\end{rems}

Figure~\ref{fig:2} illustrates the case $g_X=5$, $f_X=4$, and $p=5$,
with the same conventions as in the previous figure.  In this case,
$\pi_*:\NN_Y/\delta\NN_Y\to H^1_{dR}(X)$ is an isomorphism, as is
$\pi^*:H^1_{dR}(X)\to\NN_Y[\delta]$.  The gray zone on the right
represents the submodule $\NN^1_Y$ and the gray zone on the left
represents the quotient module $\NN^3_Y/\NN_Y^2$.  Note that the
classes $\omega_X$, $\eta_Y$, and $\omega_Y$ are not canonically
defined.

\begin{proof}[Proof of Proposition~\ref{prop:H^1(Y,T)}]
(1) By \cite[Thm.~1]{Nakajima85}, $H^0(Y,\Omega^1_Y(T))$ is free over
$k[G]$ of rank $g_X$.  The same is true of $H^1(Y,\OO_Y(-T))$ by Serre
duality.  The modified de Rham exact sequence
\[0\to H^0(Y,\Omega^1_Y(T))\to
  \Hy^1\left(Y,\OO_Y(-T)\labeledto{d}\Omega^1_Y(T)\right) \to
  H^1(Y,\OO_Y(-T))\to 0\]
shows that
$\NN_Y=\Hy^1\left(Y,\OO_Y(-T)\labeledto{d}\Omega^1_Y(T)\right)$ is
$k[G]$-free of rank $2g_X$.

(2) The quotients appearing in the statement are all isomorphic (via a
suitable power of $\delta$) to $\NN_Y[\delta]$, so it will suffice to
prove that $\NN_Y[\delta]\cong H^1_{dR}(X)$ as $\Dk$-modules.

Note that $\pi^*\OO_X(-S)=\OO_Y(-T)$ and
$\pi^*\Omega^1_X(S)=\Omega^1_Y(T)$.
The exact sequence of low degree terms of the Hochschild-Serre
spectral sequence for $Y\to X$ and $\OO_X(-S)$ yields an isomorphism
\[\pi^*:H^1(X,\OO_X(-S))\isoto H^1(Y,\OO_Y(-T))[\delta],\]
and since $Y\to X$ is unramified, it is clear that we have an
isomorphism
\[\pi^*:H^0(X,\Omega^1_X(S))\isoto H^0(Y,\Omega^1_Y(T))[\delta].\]
Using the modified de Rham sequences of $\OO_X(-S)\labeledto{d}\Omega^1_X(S)$
and $\OO_Y(-T)\labeledto{d}\Omega^1_Y(T)$ shows that $\pi^*$ induces an
isomorphism
\[\pi^*:\Hy^1\left(X,\OO_X(-S)\labeledto{d}\Omega^1_X(S)\right)
  \isoto\Hy^1\left(Y,\OO_Y(-T)\labeledto{d}\Omega^1_Y(T)\right)[\delta].\]
Since $\deg S=1$, we also have $H^1(X,\OO_X(-S))\cong H^1(X,\OO_X)$,
and $H^0(X,\Omega^1_X(S))\cong H^0(X,\Omega^1_X)$, so
\[H^1_{dR}(X)=\Hy^1\left(\OO_X\labeledto{d}\Omega^1_X\right)
  \cong \Hy^1\left(X,\OO_X(-S)\labeledto{d}\Omega^1_X(S)\right).\]
This completes the proof of part (2).

\begin{figure}
\begin{center}
\begin{tikzpicture}[scale=0.7]
% Y, m
  \draw[step=1cm] (-2,0) grid (-5,5);
  \draw[step=1cm] (-5,0) grid (-6,1) node[anchor=north west]
  {$\omega_Y$};
  \fill[step=1cm,color=gray!50]  (-6,1) rectangle (-5,5);
  \draw[step=1cm]  (-6,1) grid (-5,5);
% Y, ll
  \draw[step=1cm] (-1,0) grid (1,5);
%Y, et
  \draw[step=1cm] (2,0) grid (5,5);
  \draw[step=1cm] (5,4) grid (6,5) node[anchor=north east] {$\eta_Y$};
  \fill[step=1cm,color=gray!50]  (5,0) rectangle (6,4);
  \draw[step=1cm]  (5,0) grid (6,4);
%X,m
  \draw[step=1cm] (-5,-3) grid (-6,-2) node[anchor=north west] {$\omega_X$};
  \draw[step=1cm] (-2,-3) grid (-5,-2) node[anchor=south west] {$\uparrow\pi^*$};
%X, ll
  \draw[step=1cm] (-1,-3) grid (1,-2);
%X, et
  \draw[step=1cm] (2,-3) grid (5,-2) node[anchor=south east] {$\downarrow\pi_*$};
  \draw[step=1cm] (5,-3) grid (6,-2) node[anchor=north east] {$\eta_X$};
\end{tikzpicture}
\end{center}
\caption{}\label{fig:2}
\end{figure}

(3) Note that there are exact sequences of complexes of coherent sheaves on $Y$:
\[0\to\left(\OO_Y(-T)\labeledto{d}\Omega^1_Y\right)\to
  \left(\OO_Y(-T)\labeledto{d}\Omega^1_Y(T)\right)
  \to\left(\Omega^1_Y(T)/\Omega^1_Y\right)[-1]\to0,\]
  and
  \[0\to \left(\OO_Y(-T)\labeledto{d}\Omega^1_Y\right)\to
    \left(\OO_Y\labeledto{d}\Omega^1_Y\right)\to
    \OO_T\to 0.\]
  Let $\NN_Y^2:=\Hy^1\left(Y,\OO_Y(-T)\labeledto{d}\Omega^1_Y\right)$ and
  $\NN_Y^1:= \coker\left(H^0(Y,\OO_Y)\to H^0(T,\OO_T)\right)$.

Taking cohomology of the first exact sequence above, we have
\begin{equation}\label{eq:NN-dev1}
0\to\NN^2_Y\to\NN_Y\to H^0(Y,\Omega^1_Y(T)/\Omega^1_Y)
     \to H^1(Y,\Omega^1_Y)\to0,  
\end{equation}
where the surjection on the right is the sum of residues map.  Note
that $H^0(Y,\Omega^1_Y(T)/\Omega^1_Y)$ is $k[G]$-free of rank 1, and
$H^1(Y,\Omega^1_Y)$ is $k$ with the trivial $k[G]$ action.  The action
of $F$ on both of these spaces is zero, and the $V$ action is
induced by the Cartier operator on differentials.

Taking cohomology of the second exact sequence above, we have
\begin{equation}\label{eq:NN-dev2}  
0\to H^0(Y,\OO_Y)\to H^0(Y,\OO_T)\to
     \NN_Y^2\to H^1_{dR}(Y)\to0.
\end{equation}
  Note that $H^0(Y,\OO_T)$ is $k[G]$-free of rank 1, and
   $H^0(Y,\OO_Y)$ is $k$ with the trivial $k[G]$ action.  The action
   of $F$ on both of these spaces is the usual action of Frobenius,
   and the $V$ action is trivial.

The last two displayed exact sequences give the asserted filtration on
$\NN_Y$, and this completes the proof of part (3) of the proposition.
\end{proof}

\begin{rems}\label{rem:NN}\mbox{}
  \begin{enumerate}
  \item   The hypercohomology group $\NN_Y$ is closely related to the de
  Rham cohomology of the singular curve associated to $Y$ and $T$
  where $T$ is viewed as a "modulus", as in \cite[Ch.~4]{SerreAGCF}.
\item   Suppose that $g_X>1$ and choose a $k$-rational, effective, reduced
  divisor $S$ on $X$.  A canonical divisor of $X$ is
  $k$-rational and effective, and since $k$ is perfect, the underlying
  reduced divisor of a $k$-rational divisor is again $k$-rational, so
  there always exists a divisor $S$ as above with degree $\le
  2g_X-2$.  Let $T=\pi^{-1}(S)$.  Then the proof of
  Proposition~\ref{prop:H^1(Y,T)} applies essentially verbatim and
  shows that $\NN_Y$ is free over $k[G]$ and that
  \[ \NN_Y[\delta]\cong \NN_X:=
    \Hy^1\left(X,\OO_X(-S)\labeledto{d}\Omega^1_X(S)\right).\]
Thus, at the expense of enlarging $H^1_{dR}(X)$ and $H^1_{dR}(Y)$ by
certain simple \'etale and multiplicative $\D_k$-modules, we can
always arrange that the cohomology associated to $Y$ is $k[G]$-free
with subquotients isomorphic to the cohomology associated to $X$.
\item We may also choose two reduced, strictly effective divisors
  $S_1$ and $S_2$ on $X$, set $T_i=\pi^*(S_i)$ and take
  hypercohomology of the complexes
  \[\OO_X(-S_1)\to\Omega^1_X(S_2)\and
    \OO_Y(-T_1)\to\Omega^1_Y(T_2).\]
  Then the latter is a free $k[G]$-module with subquotients isomorphic
  to the former.
  \end{enumerate}
\end{rems}

\section{Proofs of Proposition~\ref{def:GG_X} and
  Theorems~\ref{thm:G-str-genl-k} and
  \ref{thm:k-point}\label{s:groups}}\label{s:first-proofs}

\begin{proof}[Proof of Propostion~\ref{def:GG_X}]
  The group scheme homomorphisms in the first sentence of
  Definition-Proposition~\ref{def:GG_X} are obtained by applying the
  Dieudonn\'e functor to the $\D_k$-module homomorphisms in part (1)
  of Proposition~\ref{prop:H^1(Y)}.  This establishes the claims in
  Proposition-Definition~\ref{def:GG_X}, and it identifies $\MM_X$ as the
  Dieudonn\'e module of $\GG_X$.
\end{proof}

\begin{proof}[Proof of Theorem~\ref{thm:G-str-genl-k}]
  (1)  Splitting the 4-term exact sequence in part (3) of
  Proposition~\ref{prop:H^1(Y)} into two parts and using the
  identification of the image of $\pi_*$ there, we obtain
\[0\to M(\Z/p\Z)\to H^1_{dR}(X)_\et\to\delta^{p-1}H^1_{dR}(Y)_\et\to 0,\]
and
\[0\to \delta^{p-1}H^1_{dR}(Y)_\et\to H^1_{dR}(Y)_\et[\delta]\to
  M(\Z/p\Z)\to 0.\]
The first of these yields an isomorphism
$\MM_{X,\et}\cong\delta^{p-1}H^1_{dR}(Y)_\et$. Using part (4) of
Proposition~\ref{prop:H^1(Y)}, we obtain the diagram
\[\xymatrix{
    0\ar[r]&M(\Z/p\Z)\ar[r]\ar@{=}[d]&
    H^1_{dR}(X)_\et\ar[r]^{\pi^*}
    &\delta^{p-1}H^1_{dR}(Y)_\et\ar[r]\ar@{=}[d]& 0\\
    0\ar[r]&M(\Z/p\Z)\ar[r]&
    H^1_{dR}(Y)_\et/\delta\ar[u]_{\pi_*}\ar[r]^{\delta^{p-1}}&
    \delta^{p-1}H^1_{dR}(Y)_\et\ar[r]\ar@{=}[u]& 0. 
}  \]
Using the isomorphism $\MM_{X,\et}\cong\delta^{p-1}H^1_{dR}(Y)_\et$
and applying the Diedonn\'e functor yields the exact
sequence~\eqref{eq:GG_Y,et,ker} of Proposition~\ref{thm:G-str-genl-k}
and an identification of it with the exact sequence \eqref{eq:GG_X,et}.  Similarly,
the second exact sequence above yields \eqref{eq:GG_Y,et,coker}.

Part~(2) of Proposition~\ref{prop:H^1(Y)} shows that the subquotients
\[\frac{\delta^iH^1_{dR}(Y)_\et}{\delta^{i+1}H^1_{dR}(Y)_\et}\quad
        \text{for $i=1\dots,p-1$}
        \and
   \frac{H^1_{dR}(Y)_\et[\delta^i]}{H^1_{dR}(Y)_\et[\delta^{i-1}]}\quad
         \text{for $i=2\dots,p$}
       \]
are all isomorphic to one another via a suitable power of $\delta$.
Since $\delta^{p-1}H^1_{dR}(Y)_\et\cong\MM_{X,\et}$, applying the
Dieudonn\'e functor yields the isomorphisms in the first sentence of
part (1) of Theorem~\ref{thm:G-str-genl-k}, and this completes the
proof of this part.

Part~(2) is equivalent to part~(1) by Cartier duality.

For part~(3), note that part~(2) of Proposition~\ref{prop:H^1(Y)}
shows that the subquotients 
\[\frac{\delta^iH^1_{dR}(Y)_{ll}}{\delta^{i+1}H^1_{dR}(Y)_{ll}}\quad
        \text{for $i=1\dots,p$}
        \and
   \frac{H^1_{dR}(Y)_{ll}[\delta^i]}{H^1_{dR}(Y)_{ll}[\delta^{i-1}]}\quad
         \text{for $i=1\dots,p$}
       \]
are all isomorphic to one another via a suitable power of $\delta$.
Moreover, parts~(1) and (3) of Proposition~\ref{prop:H^1(Y)} show that
\[\MM_{X,ll}\cong H^1_{dR}(X)_{ll}\cong H^1_{dR}(Y)_{ll}[\delta],\]
so all of the subquotients above are isomorphic to $\MM_{X,ll}$. Applying
the Dieudonn\'e functor yields part~(3) of
Theorem~\ref{thm:G-str-genl-k}.  This completes the proof of that
theorem. 
\end{proof}

\begin{proof}[Proof of Theorem~\ref{thm:k-point}]
  Let $\HH$ be the $k$ group scheme with Dieudonn\'e module $\NN_Y$ as
  defined in Proposition~\ref{prop:H^1(Y,T)}.  By part (1) of that
  proposition, $\HH$ is $G$-free, and we have equalities
  $\delta^{i}\HH=\HH[\delta^{p-i}]$ for $i=1,\dots,p$.  By part (2),
  there are canonical isomorphisms
  \[      \frac{\delta^{i}\HH}{\delta^{i+1}\HH}\cong
    \frac{\HH[\delta^{p-i}]}{\HH[\delta^{p-i-1}]}\cong
    J_X[p]
        \quad\text{for $i=0,\dots,p-1$.}\]
  Note that
\[M\left(\res_{T/S}\Z/p\Z\right)\cong H^0(Y,\OO_T)\]
where the right hand side is a $k$ vector space on which $F$
acts by the $p$-power Frobenius and $V=0$, and that
\[M\left(\res_{T/S}\mu_p\right)\cong
  H^0\left(Y,\Omega^1_Y(T)/\Omega^1_Y\right)\]
where the right hand side is a $k$ vector space on which $F=0$ and $V$
acts by the Cartier operator (which is essentially the inverse
Frobenius on residues).  Then applying the Dieudonn\'e functor to
part (3) of Proposition~\ref{prop:H^1(Y,T)} yields the exact sequences
asserted in part (3) of Theorem~\ref{thm:k-point}.  This completes the
proof of the theorem.
\end{proof}

\section{Comments on $\D_k[G]$-modules and
  examples}\label{s:comments}

\subsection{Motivation}\label{ss:motivation}
Suppose that $N$ is a finite-dimensional $k$-vector space equipped
with an action of $\D_k$ and/or $G=\Z/p\Z$.  Then $N$ is both Artinian
and Noetherian, so the Krull-Schmidt theorem holds: $N$ is the direct
sum of indecomposable submodules, and the number and isomorphism types
of the summands are uniquely determined. (See, for example,
\cite[\S3.4]{JacobsonBasicAlgebraII}.)

In the case where $N=H^1_{dR}(Y)$ or $\NN_Y$, we have complete
information on $N$ as a $k[G]$-module from part (2) of
Proposition~\ref{prop:H^1(Y)} or part (1) of
Proposition~\ref{prop:H^1(Y,T)} respectively.  If we take
$H^1_{dR}(X)$ as given, then we know the associated graded of $N$ with
respect to the two filtrations attached to the $G$-action by the
proofs of Theorems~\ref{thm:G-str-genl-k} and \ref{thm:k-point} in
Section~\ref{s:first-proofs}.

The basic question that motivates Sections~\ref{s:apps-etale} and
\ref{s:apps-ll} is this: what restrictions on $N$ as a $\D_k$-module
are imposed by the information in the preceding paragraph?  As we will
see below, this question appears to be quite difficult, and we are only
able to give satisfactory answers in a few cases.  In the rest of
this section, we explain some of the difficulties and give examples
illustrating them.

\subsection{$H^1_{dR}(Y)$ is not determined by $H^1_{dR}(X)$}

We show by example that $H^1_{dR}(Y)$ is not determined by
$H^1_{dR}(X)$, even when $k$ is algebraically closed; that is,
$H^1_{dR}(Y)$ is not determined as a $\D_k[G]$-module by its
associated graded $\D_k$-module.

Assume that $k$ is algebraically closed.  Then, as explained in the
proof of part (1) of Proposition~\ref{prop:H^1(Y)}, the data of the
cover $\pi:Y\to X$ and the isomorphism $\gal(Y/X)\cong\Z/p\Z$
determines and is determined by an element in the finite-dimensional
$\Fp$-vector spaces
\[H^1_{et}(X,\Z/p\Z)\cong H^1_{dR}(X)^{F=1}\cong H^1(X,\OO_X)^{F=1}.\]
Multiplying the element by a scalar in $\alpha\in\Fp^\times$
represents the same cover $\pi:Y\to X$ with the isomorphism
$\gal(Y/X)\cong\Z/p\Z$ multiplied by $\alpha$.  The set of unramified
covers $\pi:Y\to X$ (without the isomorphism $\gal(Y/X)\cong\Z/p\Z$)
is thus in bijection with the projective space
$\P\left(H^1(X,\OO_X)^{F=1}\right)$.  For general $k$, elements of
$\P\left(H^1(X,\OO_X)^{F=1}\right)$ correspond to $\kbar$-isomorphism
classes of covers which can be defined over $k$.  (In general, they
are represented by several distinct $k$-isomorphism classes of covers.)

Take $p=3$ and let $X$ be the degree 9 hyperelliptic curve over
$k:=\F_3$ given by
\[X:\qquad y^2 = x^9 + x^4 +x^2 + 1.\]
Then $X$ has genus $g_X = 4$ and $f_X = \nu_X = 2$, and
$H^1(X,\O_X)^{F=1}$ is 2-dimensional over $\F_p$.  The 4 unramified
covers of $X$ over $\kbar$ thus all arise from covers defined over
$k$.  For each, we choose a cover representing it given by
$Y_i:\ z^p - z = f_i$, with $f_i$ in the table below, and for each
cover $Y_i$ we record the EO-type of $\D(J_{Y_i}[p]_{ll})$ (which
determines the isomorphism class of $J_{Y_i}[p]_{ll}$ over $\kbar$,
but not in general over $k$).  For $a\in \F_3$, we also consider the
twisted curve $Y_i(a): z^p-z = f_i + a$ and record the invariant
factors of $F$ acting on $\D(J_{Y_i}(a)[p]_{et})$, which determine its
isomorphism type as a $k[F]$-module.  We emphasize that
$Y_i(a)\simeq Y_i$ over $\kbar$ for each $a$.  This list of invariant
factors will be of the form $(F-1)^{e_1}, (F-1)^{e_2},\ldots$, with
$e_1\le e_2\le\ldots$, and for ease of notation we record it simply as
$e_1,e_2,\ldots$.  Perhaps surprisingly, the isomorphism type varies
among the three possible twists over $k$.

\begin{center}
\begin{tabular}{c|c|c|c|c}
	 $Y_i:\ z^p - z = $ & EO-type of $\D(J_{Y_i}[p]_{ll})$ & \multicolumn{3}{c}{Inv. factors of $F$ on $\D(J_{Y_i(a)}[p]_{\et})$} \\
	\hline 
	                                           &                               &          $ a=0$ & $a=1$ & $a=-1$ \\
	                                           \cline{3-5}
	 $-(x^6 + x^3 + x -1)y $ & $[ 0, 0, 0, 1, 2, 3]$ & $1,1,2$  &  $1,3$  & $1,3$\\
	  $ (x^6 + x + 1)y $ & $[ 0, 0, 1, 1, 2, 3]$ & $1,1,2$  &  $1,3$  & $1,3$\\
	 $(x^3 + 1)y +1$ & $[ 0, 1, 1, 2, 3, 4]$ & $1,3$  &  $1,3$  & $1,1,2$\\
	$(x^6 - x^3 + x)y$ & $[ 0, 1, 1, 2, 3, 4]$ & $1,1,2$  &  $1,3$  & $1,3$\\
\end{tabular}
\end{center}

Note that the three group schemes $J_{Y_i}[p]_{ll}$ for $1\le i\le 3$
are pairwise non-isomorphic over $\kbar$, and that the four $k$-group
schemes $J_{Y_i}[p]$ for $1\le i\le 4$ are pairwise non-isomorphic
over $k$.  However, $J_{Y_3}[p]$ and $J_{Y_4}[p]$ become isomorphic
over $\kbar$.  Note as well that the EO-types of $J_{Y_3}[p]$ and
$J_{Y_4}[p]$ show that the conclusion of Theorem~\ref{thm:h=1-app}
need not hold when $g_X-f_X>1$.

\subsection{Extensions of $BT_1$ modules}\label{ss:extensions}
Propositions~\ref{prop:H^1(Y)} and \ref{prop:H^1(Y,T)} tell us that
the self-dual $BT_1$ module $H^1_{dR}(Y)_{ll}$ is a repeated extension
of the self-dual $BT_1$ module $H^1_{dR}(X)_{ll}$.  In this section,
we make some comments (surely well-known to experts) about how
ill-behaved such extensions may be, even when $k$ is algebraically
closed

More precisely, consider the full subcategory of the category of
$\Dk$-modules whose objects are $BT_1$ modules.  We use freely the
Kraft classification of $BT_1$ modules in terms of cyclic words on the
two letter alphabet $\{f,v\}$.  (See \cite{PriesUlmer21} for an overview.)

It is a simple exercise to check that this category is closed under
extension and quotient in the sense that if
\[0\to M_1\to M\to M_2\to0\]
is an exact sequence of $\Dk$-modules, and if two of $M,M_1,M_2$ are
$BT_1$ modules, then so is the third.
% Use the short exact sequence of complexes with middle term
% \[M\labeledto{V}M\labeledto{F}M.\]
% The middle homology is the obstruction to being a $BT_1$.
% The long exact sequence of homology gives the assertion immediately when
% $M_1$ and $M_2$ are $BT_1$'s.  For the other cases, use
% $\ker V_M=\im F_M$, $\coker F_M=\coim V$, etc., to see that the maps
% before or after the obstructions are onto/into. 

However, this category has the unfortunate property that the image and
kernel of a morphism between $BT_1$ modules need not be $BT_1$
modules.  For example, if $M_{1,1}$ is the module associated to the word
$fv$ (so one generator $e$ and one relation $Fe=Ve$), then the maps of
modules $M\to M$ determined by sending $e$ to $Fe$ has kernel and
image isomorphic to the module $\Dk/(F,V)$, and this is not a $BT_1$.

In \cite{Oort05}, there is a determination of the simple $BT_1$
modules, i.e., those that have no non-trivial $BT_1$ submodule.  It
follows by standard arguments that every $BT_1$ module is an iterated
extension of simple $BT_1$ modules.

Unfortunately, there is no Jordan-H\"older theorem here: The list of
simple $BT_1$ modules appearing in a presentation of a given $M$ as an
extension of simple $BT_1$ modules is not in general uniquely determined.

Here is an example: Let $M$ be the module associated to $f^3v^3$ (one
generator $e$ with relation $F^3e=V^3e$).  Then is is not hard to
check that $M$ is a three-fold extension of $M_{1,1}$.
% (Send $e$ to the standard generator of $M_{1,1}$.  The kernel is
% spanned by $e':=(F-V)e$, $F^2e=Fe'$, $F^3e=F^2e'$, $V^2e=Ve'$, and as
% a $\Dk$ module, is cyclic, generated by $e'$ with relation
% $F^2e'=V^2e'$.  Do this twice more to see that $M$ has simple factors
% $M_{1,1}$ three times.
On the other hand, $M$ also admit a surjection onto the module
$M_{2,1}$ corresponding to $ffv$ (one generator $a$ with relation
$F^2a=Va$) with kernel isomorphic to $M_{1,2}$, the module
corresponding to the word $fvv$ (one generator $b$ with relation
$Fb=V^2b$).  Oort's results imply that $N_{2,1}$ and $N_{1,2}$ are
simple, so we have no uniqueness of ``Jordan-H\"older'' factors.

\subsection{$\D_k[G]$-modules}
We show by example that there may be infinitely many non-isomorphic
$\D_k[G]$-modules with given $\D_k$- and $k[G]$-module structures.  We
first show that (for any $p$), the $\D_k$-module associated to the word
$f^2v^2$ has infinitely many non-isomorphic presentations as an
extension of the module associated to $fv$ by itself.  Recall
\cite[\S3.2]{PriesUlmer21} that $M(fv)$ is given by one generator
$e_1$ and one relation $Fe_1=Ve_1$, and that $M(f^2v^2)$ is given by
one generator $e_2$ and one relation $F^2e_2=V^2e_2$.  A homomorphism
$\phi:M(f^2v^2)\to M(fv)$ is determined by its value on $e_2$ which is
in turn determined by two arbitrary elements of $k$:
\[\phi(e_2)=\alpha e_1+\beta Fe_1,\]
and the surjective homomorphisms are those with $\alpha\neq0$.  The
kernel of a surjective $\phi$ is the subspace spanned by
$\alpha^{1/p}Fe_2-\alpha^pVe_2$ and $F^2e_2$.  One checks that this
subspace is a $\D_k$-module isomorphic to $M(fv)$.  Note that the
kernels of distinct $\phi$ are in general distinct subspaces of
$M(f^2v^2)$; in fact, they are in general not even conjugate by
$\aut_{\D_k}(M(f^2v^2))\cong\F_{p^4}\times k^3$.

Now assume that $p=2$.  For each choice of surjective $\phi$ as above,
we may equip $M=M(f^2v^2)$ with the structure of a free $R=k[G]$ module of
rank 2 in such a way that $M[\delta]=\ker\phi$.  (See
Lemma~\ref{lemma:p=2-R-structure} below.)  This shows that there
are infinitely many non-isomorphic $\D_k[G]$-modules with the same
underlying $R$-module  (namely $R^2$) and the same underlying $\D_k$
module (namely $M(f^2v^2)$).

\section {Applications to the \'etale part of $J_Y[p]$}\label{s:apps-etale}
We turn to a consideration of the \'etale part of $J_Y[p]$.  When $k$
is algebraically closed, $J_Y[p]_\et$ is completely determined by the
isomorphism \eqref{eq:DSet}.  As we will see below, the situation is
more interesting when $k$ is only assumed to be perfect.

\begin{proof}[Proof of Proposition~\ref{prop:splitting}]
Applying the Dieudonn\'e functor to parts (3) and (4) of
Proposition~\ref{prop:H^1(Y)} yields an exact sequence
\[0\to\Z/p\Z\to J_Y[p]_\et/\delta
  \overset{\delta^{p-1}}{\xrightarrow{\hspace*{1cm}}}
    J_Y[p]_\et[\delta]\to\Z/p\Z\to 0,\]
and an identification of $\im\delta^{p-1}$ with $\GG_X$ via the
isomorphism $\pi^*:J_X[p]_\et\to J_Y[p]_\et[\delta]$.  Splitting this
exact sequence into two short exact sequences yields sequences
\eqref{eq:GG_Y,et,ker} and \eqref{eq:GG_Y,et,coker}.  

For part (1), assume that the sequence \eqref{eq:GG_Y,et,ker} splits.
Then we have an inclusion
\[\Z/p\Z\into J_Y[p]_\et[\delta]\into J_Y[p]_\et\]
whose image does not lie in $\im\delta^{p-1}$.  By part (2) of
Theorem~\ref{prop:H^1(Y)}, the image therefore does not lie in the
image of $\delta$.  This shows that the quotient $\QQ$ defined by the
exactness of
\begin{equation}\label{eq:1.5-splits}
0\to\Z/p\Z\to J_Y[p]_\et\to\QQ\to 0  
\end{equation}
is $G$-free.

Conversely, if we have the exact sequence \eqref{eq:1.5-splits} with
$\QQ$ assumed to be $G$-free, then the image of $\Z/p\Z\to J_Y[p]$ is
killed by $\delta$ and not in the image of $\delta$, so it splits
\eqref{eq:GG_Y,et,ker}.  This completes the proof of part (1) of the
proposition.

For part (2), assume that the sequence \eqref{eq:GG_Y,et,coker} splits.
Then we have a surjection
\[J_Y[p]_\et\onto J_Y[p]_\et/\delta\onto\Z/p\Z\]
whose kernel maps surjectively to $\GG_X$ via $\delta^{p-1}$.  This
shows that the subgroup $\KK$ defined by the 
exactness of
\begin{equation}\label{eq:1.6-splits}
0\to\KK\to J_Y[p]_\et\to\Z/p\Z\to 0  
\end{equation}
is $G$-free.

Conversely, if we have the exact sequence \eqref{eq:1.6-splits} with
$\KK$ assumed to be $G$-free, then the surjection
$J_Y[p]_\et\to\Z/p\Z$ factors through $J_Y[p]_\et/\delta$, so it
splits \eqref{eq:GG_Y,et,coker}.  This completes the proof of part (2)
of the proposition.

For part (3), if the exact sequences \eqref{eq:GG_Y,et,ker} and
\eqref{eq:GG_Y,et,coker} both split, then we have the sequences
\eqref{eq:1.5-splits} and \eqref{eq:1.6-splits}.  It then follows
easily from part (2) of Theorem~\ref{prop:H^1(Y)} that the composed
maps
\[\Z/p\Z\to J_Y[p]_\et\to\Z/p\Z\]
and
\[\KK\to J_Y[p]_\et\to\QQ\]
are isomorphisms, so $J_Y[p]_\et$ is isomorphic to the direct sum of
$\Z/p\Z$ and a $G$-free group.  (Moreover, sequences
\eqref{eq:1.5-splits} and \eqref{eq:1.6-splits} both split.)  The
converse is straightforward.

This completes the proof of Proposition~\ref{prop:splitting}
\end{proof}

\begin{s-example}\label{ex:splitting}
  In general, the splitting of exact sequences \eqref{eq:GG_Y,et,ker}
  and \eqref{eq:GG_Y,et,coker} appear to be independent conditions:
  Table \ref{UnrTable} provides several examples of \'etale
  $\Z/3\Z$-covers of genus 3 hyperelliptic curves over $\F_3$ with
  $f_X=2$ exhibiting that all 4 splitting possibilities indeed occur.
  (The additional notations $d_1$, $d_2$, and $\mu$ are explained in
  Example~\ref{ex:extension-bounds} below.)  Note, however, that in
  certain special situations, there are implications: for example when
  $k$ is finite and $\nu_X=1$, the splitting of \eqref{eq:GG_Y,et,ker}
  implies that of \eqref{eq:GG_Y,et,coker} by Theorem
  \ref{thm:et-inequalities} \eqref{thm:et-inequalities:part4}.

\end{s-example}

\begin{center}
\begin{table}
\begin{tabular}{c|c|c|c|c|c|c|c|c}
	$X: y^2 = $ &  $\nu_X$ & $Y: z^3-z = $ & $\nu_Y$ & $\eqref{eq:GG_Y,et,ker}$ split? &  $\eqref{eq:GG_Y,et,coker}$ split? & $d_1$ & $d_2$  & $\mu$  \\
	\hline 
	$x^7 + x^5 + x$ &1 & $(x^4 + x^2 + x)y$ & 1 & yes & yes & 2 & 1 & 3 \\
	$-x^7 - x^3 + x $ & 1 & $xy$ & 1 & yes & yes & 2 & 3 & 2 \\ 	
	$-x^7 - x^3 - 1 $ & 1 & $xy-1$ & 1 & yes & yes & 2 & 3 & 1 \\ 	
	$x^7 + x^6 + x^4 + x$ & 1 &  $(x + 1)y$ & 2 & no & no  & 3 & 1&  \\
	$x^7 + x^6 + x^4 + 1$ & 1 & $(x + 1)y + 1$ & 2 & no & yes & 3 & 1&  \\
	$-x^7 + x^6 + x^4 + x$ & 1 &  $(-x + 1)y$ & 3 & no & yes & 3 & 1 & \\
	$x^8 -x^6 + x^4 - x^2 + 1$ & 2 & $ (x^2 - 1)y - x^6 + x^2 - 1 $ & 2 & yes & yes & 1 & 3 &  \\
	$ -x^7 + x^6 + x^2 + x$ & 2 & $(-x + 1)y$ & 3 & yes & no & 1 & 3 & \\
	$-x^7 + x^6 + x^2 + x$ & 2 & $(-x^4 + x^3 + x - 1)y$ & 3 & yes & yes & 1 & 3 & \\
	$ -x^7 + x^6 + x^2 + x$ & 2 & $(-x^4 + x^3)y$ & 4 & yes & yes & 1 & 1& 
\end{tabular}
	\caption{$p=3$, $f_X=2$}
	\label{UnrTable}
	\end{table}
\end{center}

In some of the proofs below, it will be convenient to use the language
of Galois representations.  Recall (e.g., using
\cite[\S5.8]{PoonenRPV}) that there is an equivalence of categories
between \'etale $p$-group schemes over $k$ and representations of
$\gal(\kbar/k)$ on finite-dimensional $\Fp$-vector spaces.  The
representation associated to a group $\GG$ is the $\Fp$-vector space
$\VV:=\GG(\kbar)$ equipped with the natural action of $\gal(\kbar/k)$.
In particular, when $k$ is finite, $\gal(\kbar/k)$ is pro-cyclic, and
$\GG$ is determined by a single endomorphism of $\VV$, namely
Frobenius.

\begin{proof}[Proof of Theorem~\ref{thm:et-inequalities}]
  First note that if
\[0\to\GG_1\to\GG_2\to\GG_3\to0\]
  is an exact sequence of $p$-torsion group schemes over $k$, then
  \[ \nu(\GG_1)\le\nu(\GG_2)\le\nu(\GG_1)+\nu(\GG_3).\]
  If the sequence splits, then $\nu(G_2)=\nu(\GG_1)+\nu(\GG_3)$.
  
  By part (1) of Theorem~\ref{thm:G-str-genl-k},
  \[J_X[p]_\et\cong J_Y[p]_\et[\delta]\subset J_Y[p]_\et,\]
  so we have $\nu_X\le\nu_Y$.

  Again by part (1) of Theorem~\ref{thm:G-str-genl-k}, we have exact
  sequences
  \[0\to J_Y[p]_\et[\delta^{i-1}]\to
    J_Y[p]_\et[\delta^{i}]\to\GG_X\to0\]
  for $i=2,\dots,p$.

  Applying the observation in the first paragraph of the proof and
  induction on $i$, we find that $\nu_Y\le\nu_X+(p-1)\nu(\GG_X)$.
  Since $\GG_X$ is a subgroup scheme of $J_X[p]$, $\nu(\GG_X)\le\nu_X$
  and we have $\nu_Y\le p\nu_X$.  This establishes part (1).

  If \eqref{eq:GG_X,et} splits, then
  $\nu(\GG_X)=\nu_X-1$ and we find that $\nu_Y\le p(\nu_X-1)+1$.  This
  yields part (2).

  For part (3), if $k$ is algebraically closed, then $\nu_X=f_X$ and
  $f_X\ge 1$ since we have assumed $X$ has an \'etale $\Z/p\Z$ cover.
  Now assume that $k$ is finite.  Note that by
  Definition-Proposition~\ref{def:GG_X}, the representation associated
  to $J_X[p]_\et$ has the trivial representation as a quotient.  This
  is equivalent to saying that the action of Frobenius on
  $\VV=J_X[p](\kbar)$ has 1
   as an eigenvalue.  It follows that $J_X[p](k)$ is non-trivial, i.e.,
  $\nu_X\ge1$.

  If $k$ is finite, \eqref{eq:GG_Y,et,ker} is split, and
  \eqref{eq:GG_Y,et,coker} is non-split, then we claim that $\Fr_k-1$
  can not act bijectively on $\GG_{X,\et}(\kbar)$, where $\Fr_k$
  generates $\Gal(\kbar/k)$.  Indeed, assuming to the contrary that
  $\Fr_k-1$ is bijective on $\GG_{X,\et}(\kbar)$, one sees via the
  snake lemma that in the exact sequence of $\kbar$-points associated
  to \eqref{eq:GG_Y,et,coker}, the image of $\Fr_k-1$ on
  $(J_Y[p]_{\et}/\delta)(\kbar)$ projects isomorphically onto the
  quotient $\GG_{X,\et}(\kbar)$; the inverse of this isomorphism
  splits \eqref{eq:GG_Y,et,coker}, contradicting our hypothesis.  We
  conclude that $\Fr_k$ must have a nontrivial fixed vector on
  $\GG_{X,\et}(\kbar)$.  Since \eqref{eq:GG_X,et} (equivalently
  \eqref{eq:GG_Y,et,ker}) splits by hypothesis, we deduce that the
  space of $\Fr_k$-fixed vectors in $J_X[p]_{\et}(\kbar)$ is at least
  $2$-dimensional, {\em i.e.}~$\nu_X\ge 2$.  This completes the proof
  of part (4) of the theorem.
\end{proof}

\begin{s-example}\label{nubound:sharp} 
  Table \ref{UnrTable} shows that the bounds on $\nu_Y$ in Theorem
  \ref{thm:et-inequalities}
  \eqref{thm:et-inequalities:part1}--\eqref{thm:et-inequalities:part2}
  are sharp: if \eqref{eq:GG_Y,et,ker} is non-split, we must have
  $\nu_X=1$ since $f_X=2$ throughout the table (indeed, the
  alternative is $\nu_X=2=f_X$, in which case $J_X[p]_{\et}$ would be
  completely split, implying the splitting of \eqref{eq:GG_Y,et,ker}).
  This forces $\nu_Y > 1$ by Theorem \ref{thm:f_X=2} (B) and
  Proposition \ref{prop:splitting} (3), which gives the bounds
  $1=\nu_X < \nu_Y \le 3$ in this situation. Lines 4--6 of the table
  show that both possibilities $\nu_Y=2,3$ indeed occur.
 
  When
  $\eqref{eq:GG_Y,et,ker}$ is split, we have
  $\nu_X \le \nu_Y \le 3\nu_X-2$; lines 1--3 show that the unique possibility $\nu_Y=1$
  indeed occurs when $\nu_X=1$, while
  lines 7--10 show that all three
  possibilities $\nu_Y = 2,3,4$ occur when $\nu_X=2$.  
\end{s-example}

\begin{s-example}\label{ex:nu_X=0}
  In part (3), if we do not assume that $k$ is finite, $\gal(\kbar/k)$
  may no longer be procyclic, and it may have $\Fp$-representations
  with trivial quotients, but no trivial subrepresentations.  Here is
  an example of such a representation.

  Take $p=3$, let $F=\Fp(t)$, let $F'=\Fp(v)$, and embed $F\into F'$
  by $t\mapsto (v^3-v)^2$.  Then one checks that $F'/F$ is Galois with
  group $S_3$.  Let $k$ and $k'$ be the perfections of $F$ and $F'$
  respectively.  Then we have
  \[\gal(\kbar/k)\onto\gal(k'/k)\cong S_3.\]

  Let $\gamma_1$ and $\gamma_2$ be two involutions generating $S_3$,
  and let $S_3$ act on $\VV:=\F_3^2$ by the matrices
  \[\gamma_1\mapsto\psmat{-1&1\\0&1}\and
    \gamma_2\mapsto\psmat{-1&0\\0&1}.\]

  Then one sees easily that the resulting representation of
  $\gal(\kbar/k)$ has the trivial representation as a quotient, but it
  does not have the trivial representation as a sub.  Thus the
  corresponding group scheme $\GG$ admits a surjection to $\Z/p\Z$,
  but has $\nu(\GG)=0$.
\end{s-example}

To prove Theorems~\ref{thm:et-extensions} and \ref{thm:f_X=2}, we will
use an approach via coordinates, so we begin with some preliminaries
on $\Fp[G]$-modules.

Let $S=\Fp[G]\cong\Fp[\delta]/(\delta^p)$, and for $i=1,\dots,p$, let
$W_i=S/(\delta^i)$ be the $i$-dimensional indecomposable module over
$S$.  Write $a\mapsto\overline a$ for the natural reduction
homomorphism $S\to k$ (the quotient modulo $\delta$).  We have natural
identifications
$\Hom_S(W_p,W_p)\cong S$,
$\Hom_S(W_p,W_1)\cong \Fp$,
$\Hom_S(W_1,W_p)\cong \Fp$,
$\Hom_S(W_1,W_1)\cong \Fp$.
Under these identifications, the maps
\[\Hom_S(W_p,W_p)\to\Hom_S(W_p[\delta],W_p[\delta])=\Hom_S(W_1,W_1)\]
and
\[\Hom_S(W_p,W_p)\to\Hom_S(W_p/\delta,W_p/\delta)=\Hom_S(W_1,W_1)\]
are both the reduction map $S\to k$.  The restriction map
\[\Hom_S(W_p,W_1)\to \Hom_S(W_p[\delta],W_1)=\Hom_S(W_1,W_1)\]
is zero, as is the reduction map
\[\Hom_S(W_1,W_p)\to \Hom_S(W_1,W_p/\delta)=\Hom_S(W_1,W_1).\]
% Finally, the composition map
% \[\Hom_S(W_p,W_1)\times \Hom_S(W_1,W_p)\to \Hom_S(W_1,W_1)\qquad
%   (\alpha,\beta)\mapsto \alpha\circ\beta\]
% is zero, whereas the composition map
% \[\Hom_S(W_1,W_p)\times \Hom_S(W_p,W_1)\to \Hom_S(W_p,W_p)\qquad
%   (\beta,\alpha)\mapsto \beta\circ\alpha\]
% is the usual multiplication in $\Fp$ composed with the map $\Fp\to S$
% which sends 1 to $\delta^{p-1}$.  (There are other composition maps to
% identify, but they are straightforward and left to the reader.)

Now let $\VV_X$ and $\VV_Y$ be the representations of
$\gal(\kbar/k)$ associated to $J_X[p]_\et$ and
$J_Y[p]_\et$:
\[\VV_X=J_X[p]_\et(\kbar)\and \VV_Y=J_Y[p]_\et(\kbar).\]
By part (2) of Proposition~\ref{prop:H^1(Y)}, we have a
(non-canonical) isomorphism of $S$-modules
\[\VV_Y\cong W_1\oplus W_p^{f_X-1}.\]
Choosing such an isomorphism, we may represent the action of
$\phi\in\gal(\kbar/k)$ by a matrix of the form
\begin{equation}\label{eq:phi-matrix}
\left(
    \begin{array}{c|c}
      a_0&\alpha_1\,\cdots\,\alpha_r\\
      \hline
      \beta_1&\\
      \vdots&A=(a_{ij})\\
      \beta_r&
    \end{array}
  \right),  
\end{equation}
where $r=f_X-1$, $a_0\in\Hom_S(W_1,W_1)$,
$\alpha_j\in\Hom_S(W_p,W_1)$, $\beta_i\in\Hom_S(W_1,W_p)$,
$a_{ij}\in \Hom_S(W_p,W_p)$, and $i$ and $j$ run from 1 to $r$.

Using the observations above on $S$-homomorphisms, we find that the
induced action of $\phi$ on $\VV_Y[\delta]\cong \VV_X$ is given by the
matrix
\begin{equation}\label{eq:phi-on-ker}
 \left(
    \begin{array}{c|c}
      a_0&0\,\cdots\,0\\
      \hline
      \beta_1&\\
      \vdots&\overline A=(\overline a_{ij})\\
      \beta_r&
    \end{array}
  \right),
\end{equation}
and the induced action of $\phi$ on $\VV_Y/\delta$ is given by the matrix
\[\left(
    \begin{array}{c|c}
      a_0&\alpha_1\,\cdots\,\alpha_r\\
      \hline
      0&\\
      \vdots&\overline A=(\overline a_{ij})\\
      0&
    \end{array}
  \right).\]
The block triangular structure of the last two displayed matrices
reflects the exact sequences \eqref{eq:GG_Y,et,ker} and
\eqref{eq:GG_Y,et,coker} (i.e., the fact that $\Z/p\Z$ is a quotient
of $\VV_Y[\delta]$ and a sub of $\VV_Y/\delta$), and we find that $a_0=1$.
Moreover, \eqref{eq:GG_Y,et,ker} splits if and only if we may choose
the isomorphism $\VV_Y\cong W_1\oplus W_p^r$ such that the $\beta_i$ all
vanish, and \eqref{eq:GG_Y,et,coker} splits if and only if we may choose
the isomorphism such that the $\alpha_j$ all
vanish.  (This gives an alternate proof of
Proposition~\ref{prop:splitting}.)

\begin{lemma}\label{lemma:1-units}
  Define $U$, the 1-units of $\GL_r(S)$, by the exactness of
  \[0\to U\to \GL_r(S)\to\GL_r(\Fp)\to 0,\]
  where the surjection is $(a_{ij})\mapsto (\overline a_{ij})$.  Then
  the group $U$ has exponent $p$.
\end{lemma}

\begin{proof}
  Write an element of $U$ in the form $I+\Delta M$ where $\Delta$ is
  the diagonal matrix with all diagonal entries equal to $\delta$ and where
  $M\in M_r(S)$.  Since $I$ and $\Delta$ are in the center of the
  characteristic $p$ ring $M_r(S)$, we have
  \[\left(I+\Delta M\right)^p=I+\Delta^pM^p=I.\]
\end{proof}

With these preliminaries, we are ready to prove the two remaining
theorems about $J_Y[p]_\et$.

\begin{proof}[Proof of Theorem~\ref{thm:et-extensions}]
  For part (1), note that if $\phi$ has matrix of the shape
  \eqref{eq:phi-on-ker} with $a_0=1$, then by an inductive argument,
  $\phi^n$ has matrix 
\[\left(
    \begin{array}{c|c}
      1&0\,\cdots\,0\\
      \hline
      n\beta_1&\\
      \vdots&\overline A^n\\
      n\beta_r&
    \end{array}
  \right).
\]
Thus there is a finite Galois extension $k'$ of $k$ with $\gal(k'/k)$
of exponent $p$ such that the representation of $\gal(\kbar/k')$ on
$\VV_X=\VV_Y[\delta]$ has the trivial representation as a direct factor,
i.e., such that the sequences \eqref{eq:GG_X,et} and
\eqref{eq:GG_Y,et,ker} split over $k'$.

For part (2), to say that $J_X[p]_\et$ is completely split is to say
that for every $\phi\in\gal(\kbar/k)$, the matrix of $\phi$ is of
the form \eqref{eq:phi-matrix} with $a_0=1$, $\beta_i=0$ for all $i$,
and $\overline A=I$, i.e., with $A\in U$.  Lemma~\ref{lemma:1-units}
then shows that $\phi^p$ acts trivially on $\VV_Y$.  This shows
that there is a  a finite Galois extension $k'$ of $k$ with $\gal(k'/k)$
of exponent $p$ such that $J_Y[p]_\et$ is completely split over $k'$.

This completes the proof of the theorem.
\end{proof}

\begin{s-example}\label{ex:extension-bounds}
	Table \ref{UnrTable} (in which each base curve $X$ has $f_X=2$) illustrates that
	the degree bounds of Theorems \ref{thm:et-extensions} and \ref{thm:f_X=2} are sharp:
	in the antepenultimate column we have listed the degree $d_1$ of the unique 
	minimal extension $k_1/k$ with the property that $J_X[p]_{\et}$ is completely split
	over $k_1$.  The penultimate column lists the degree $d_2$ of the minimal 
	extension $k_2 / k_1$ over which $J_Y[p]_{\et}$ is completely split.
	When $\nu_X=\nu_Y=1$, the final column gives the positive integer $\mu$
	specified in (1a) of Theorem \ref{thm:f_X=2}; note that all possibilities indeed occur.
	The first three lines of the table provides examples in which the extension $k'$
	guaranteed by (1a) of Theorem  \ref{thm:f_X=2} is as large as possible; in the first line $k''=k$, while
	in the second and third lines $k''$ is the unique degree $p=3$ extension of $k$. 
\end{s-example}

\begin{proof}[Proof of Theorem~\ref{thm:f_X=2}]
  For part (A), if $f_X=1$, then in equation~\eqref{eq:GG_X,et},
  $\GG_X=0$ and $J_X[p]_\et\cong\Z/p\Z$.  It then follows from part
  (1) of Theorem~\ref{thm:G-str-genl-k} that $J_Y[p]_\et\cong\Z/p\Z$.

  Now consider parts (B) and (C).  The alternatives are mutually exclusive, and
  they exhaust the possibilities, so it will suffice to verify the
  additional claims in each case.  Note that the subgroup
  $\GG_{X,\et}$ defined by exact sequence \eqref{eq:GG_X,et}
  corresponds to a line in $\VV_X$ which is invariant under
  $\gal(\kbar/k)$.  Let $\phi\in\gal(\kbar/k)$ be Frobenius, and let
  $\rho\in\Fp^\times$ be the eigenvalue of $\phi$ on this line.

  (Case (1a)) If $\rho\neq1$, then $p>2$ since $\F_2^\times=\{1\}$.
  Moreover, both \eqref{eq:GG_Y,et,ker} and \eqref{eq:GG_Y,et,coker}
  are split (by taking the kernel or image of a high power of
  $\phi-\rho$).  Thus, we may choose the isomorphism
  $\VV_Y\cong W_1\oplus W_p$ so that $\phi$ has the shape
  \[\left(
    \begin{array}{c|c}
      1&0\\
      \hline
      0&a
    \end{array}
  \right),\]
where $a\in S$ has $\overline a=\rho$.  It is then clear that
$\nu_X=\nu_Y=1$ and we are in case (1a).  The group scheme $\QQ$ in
the statement is the one corresponding to the representation of
$\gal(\kbar/k)$ on $W_p$ with Frobenius acting by $a$.  Since
$\overline a^{p-1}=\rho^{p-1}=1$, over an extension $k'$ of $k$ of
degree dividing $p-1$, $\phi$ acts on $W_p$ via a 1-unit, so has a
non-trivial space of invariants.  This shows that $|\QQ(k')|=p^\mu$
with $1\le \mu\le p$.  Since $a$ is $\rho$ times a 1-unit,
Lemma~\ref{lemma:1-units} shows that
$a^{p}=\rho^p=\rho\in k\subset S$, so that over the extension $k''$ of
degree $p$, $\rho$ acts on $W_p$ by $\rho$, and $\GG\cong(\GG')^p$ for
a rank 1, non-split group scheme $\GG'$.  Finally, $\phi^{p(p-1)}$
acts trivially on $\VV_Y$, so $J_Y[p]_\et$ is completely split over an
extension of degree dividing $p(p-1)$.

(Case (1b), $p>2$, and Case (1), $p=2$)
Next, suppose that $\rho=1$ and \eqref{eq:GG_X,et} is not split.
Then $\GG_X\cong\Z/p\Z$ and $\nu_X=1$.  If $p=2$, this is enough to
conclude that we are in case (1).  
Since $J_X[p]_\et$ is not completely split, $J_Y[p]_\et$ is also not
completely split and $\nu_Y<p+1$. 
The action of $\phi$ is given
by a matrix of the form
  \[\left(
    \begin{array}{c|c}
      1&\alpha\\
      \hline
      \beta&a
    \end{array}
  \right),\]  
where $\beta\neq0$.   If $p>2$, we
consider the matrix of $\phi$ with respect to a suitable $k$-basis
of $W_1\oplus W_p$, namely 1 for $W_1$ and
$\delta^{p-1},\delta^{p-2},\dots,1$ for $W_p$.  Then $\phi$ takes the
form
  \[\left(
    \begin{array}{c|cccc}
      1&0&\dots&0&\alpha\\
      \hline
      \beta&1&*&*&*\\
      0&0&1&*&*\cr
       \vdots&\vdots&\ddots&1&*\cr
       0&0&\dots&0&1                       
    \end{array}
  \right).\]  
It is then visible that $\phi-1$ has rank $<p$, so $\nu_Y>1$.  Thus
we are in case (1b).  For any $p$, an inductive argument shows that
$\phi^n$ has 
matrix
  \[\left(
    \begin{array}{c|c}
      1&n\alpha\\
      \hline
      n\beta&(1+2+\cdots+(n-1))\beta\alpha +a^n
    \end{array}
  \right).\]
Applying Lemma~\ref{lemma:1-units} shows that if $p>2$,
then $J_Y[p]_\et$ is completely split over the extension of $k$ degree
$p$, and if $p=2$, then $J_Y[p]_\et$ is completely split over the
extension of $k$ degree dividing $4$.

(Case (2)) Finally, if $\rho=1$ and \eqref{eq:GG_X,et} splits, then
$\nu_X=2$, $J_X[p]_\et$ splits completely, and we are in case (2) (for
any $p$).  The conclusions there follow from
Theorem~\ref{thm:et-inequalities} and part (2) of
Theorem~\ref{thm:et-extensions}.  We can say a bit more about the
structure of $J_Y[p]_\et$ in this case: By part (1) of
Proposition~\ref{prop:splitting}, there is an exact sequence
\[0\to\Z/p\Z\to J_Y[p]_\et\to\QQ\to0,\]
where $\QQ$ is $G$-free of rank 1.  We have $\nu(\QQ)\ge1$, and both
the extension above and the group scheme $\QQ$ split completely over
an extension of degree dividing $p$.

This completes the proof of Theorem~\ref{thm:f_X=2}.
\end{proof}

\begin{s-example}\label{Example:fXeq2}
  Taking into account Theorem \ref{thm:et-inequalities}
  \eqref{thm:et-inequalities:part2}, which implies that when $k$ is
  finite and $\nu_X=1$, the splitting of \eqref{eq:GG_Y,et,ker}
  implies that of \eqref{eq:GG_Y,et,coker}, we see that Table
  \ref{UnrTable} exhibits that all possibilities specified by Theorem
  \ref{thm:f_X=2} indeed occur when $p=3$ and $k=\F_3$.
\end{s-example}

\subsection{Dependence of $\HH$ and $\NN_Y$ on
  $S$}\label{ss:S-dependence}
The group scheme $\HH$ in Theorem~\ref{thm:k-point} and its
Dieudonn\'e module $\NN_Y$ (analyzed in
Proposition~\ref{prop:H^1(Y,T)}) apparently depend on the choice of a
rational point $S$.  Indeed, the subquotients
$\Ker\left(\res_{T/S}\Z/p\Z\to\Z/p\Z\right)$ and
$\coker\left(\mu_p\to\res_{T/S}\mu_p\right)$ of $\HH$ and their
Dieudonn\'e modules
\[\coker\left(k=H^0(Y,\OO_Y)\to H^0(Y,\OO_T)\right)\and
  \Ker\left(H^0(Y,\Omega^1_Y(T)/\Omega^1_Y) \to
    H^1(Y,\Omega^1_Y)=k\right)\]
visibly depend on whether $S$ splits in $Y$ (and more precisely on the
class of $T=\pi^{-1}(S)$ in $H^1(S,\Z/p\Z)$).

\begin{prop}
  If $k$ is algebraically closed, then the isomorphism class of
  $\NN_Y$ as a $\D_k[G]$-module is independent of the
  choice of the rational point $S$.
\end{prop}

\begin{proof}
  The exact sequences \eqref{eq:NN-dev1} and \eqref{eq:NN-dev2} show
  that the local-local part of $\NN_Y$ is independent of $S$ (without
  any hypothesis on $k$).  If $k$ is algebraically closed, the \'etale
  part of $\NN_Y$ is completely split as a $\D_k$-module (isomorphic to
  $M(\Z/p\Z)^{pf_X}$) and by Proposition~\ref{prop:H^1(Y,T)}, it is
  $G$-free of rank $f_X$, so it is isomorphic to
  \[M(\Z/p\Z)^{f_X}\tensor\Fp[G],\]
  and is thus independent of $S$.  The same follows for the
  multiplicative part by Cartier duality.  Since $\NN_Y$ is the direct
  sum of its \'etale, multiplicative, and local-local parts, this
  establishes the proposition.
\end{proof}

\begin{s-example}\label{NY:Dependence}
  Surprisingly, when $k$ is finite, $\NN_Y$ depends on $S$, even when
  $S$ splits in $Y$.  To emphasize the subtlety of this dependence on
  $S$ and its fiber $T:=\pi^{-1}(S)$ in $Y$, let us write $\NN_Y(T)$
  in place of $\NN_Y$.  Consider the smooth projective genus 5
  hyperelliptic curve over $k=\F_3$ given by the affine equation
$$
X:\quad	y^2 + x^{12} + x^{10} -x^9 + x^6 + x^4 -x^2 + x -1 = 0.
$$
This curve has $a$-number 1 and arithmetic $p$-rank $\nu_X=2$,
so in particular has two independent unramified $\Z/3\Z$-covers.
One such cover $Y$ is given by the Artin--Schreier equation $z^3-z=f$
where $f$ is the rational function
\begin{multline*}
f= \frac{x^9 - x^7 - x^6 + x^5 + x^4 + x^3 + x^2 - x + 1}{x^{15}}y  \\
    - \frac{x^{15} + x^{14} - x^{13} + x^{12} + x^{11} - x^{10} + x^9 + x^6 - x^3 +1}{x^{15}}. 
\end{multline*}
The projective curve $X$ has exactly six $k$-rational points $(x,y)$:
$$
	S_1:= (0, -1), S_2:=(0, 1), S_3:=(-1, -1), S_4:=(-1, 1), S_5:=(1, -1), S_6:=(1, 1);
$$
note that the point at infinity on $X$ is of degree 2.
Let $\pi:Y\rightarrow X$ be the covering map and $T_i:=\pi^{-1}S_i$ be the fiber in $Y$ over $S_i$.
Each $T_i$ is gives a $\Z/3\Z$-torsor (for the \'etale topology) over $k$, so gives a class in $H^1_{\et}(\Spec(k),\Z/p\Z)$.
Via the canonical identifications of abelian groups
$$
	H^1_{\et}(\Spec(k),\Z/p\Z) \simeq H^1(\Gal(\overline{k}/k),\Z/p\Z) \simeq k/\wp k =\Z/3\Z
$$
each torsor $T_i$ gives a class $[T_i]\in \Z/3\Z$.  For example, 
$[T_1]=0=[T_3]$ as each of $T_1$ and $T_3$ consist of 3 distinct 
$k$-rational points of $Y$, whereas $[T_i]=\pm 1$ for $i=2,4,5,6$
since for these values of $i$, the fiber $T_i$ is a single degree 3 point on $Y$.

Using {\sc Magma}, we compute the matrix of $V$ acting on the spaces 
$H^0(\Omega^1_Y(T_i))$ for $i=1,\ldots ,6$.  We also compute the action of the Artin--Schreier automorphism of $Y\rightarrow X$
on the residue field of $T_i$, which determines $[T_i]$, and obtain the following  table:

\begin{center}
\begin{tabular}{r||rrrrrr}
	$i$ & 1 & 2 & 3 & 4 & 5 & 6\\
	\hline
	$[T_i]$ & 0 & 1 & 0 & 1 & -1 & -1 \\
	$\dim \ker (V-1)$ & 4& 3& 4& 3 & 3 & 3\\
		$\dim \ker (V-1)^3$  & 9 & 9 & 8 & 8 & 8 & 8  \\
\end{tabular}
\end{center}
It follows from this that the four $k[V]$-modules
$H^0(\Omega^1_Y(T_i))=\NN_Y(T_i)[F]$ 
for $1\le i\le 4$ are pairwise non-isomorphic.  For $i=1,3$ we have short
exact sequences of $k[V]$-modules
\begin{equation*}
	\xymatrix{
		0 \ar[r] & {H^0(Y,\Omega^1_Y)} \ar[r] & {H^0(Y,\Omega^1_Y(T_i))} \ar[r] & \left(\displaystyle\frac{k[V]}{(V-1)}\right)^2 \ar[r] & 0
	}
\end{equation*}
while for $i=2,4$ we have short exact sequences
\begin{equation*}
	\xymatrix{
		0 \ar[r] & {H^0(Y,\Omega^1_Y)} \ar[r] & {H^0(Y,\Omega^1_Y(T_i))} \ar[r] & \displaystyle\frac{k[V]}{(V-1)^2} \ar[r] & 0
	}
\end{equation*}
Noting that the kernel of $V-1$ on $H^0(\Omega^1_Y)$ has dimension 3,
our computations show that the above exact sequences are not only {\em non}-split for $1\le i\le 4$, 
but that the two {\em extension} classes of $k[V]$-modules provided by $H^0(\Omega^1_Y(T_i))$
for $i=1,3$ (respectively $i=2,4$) are non-isomorphic!  This is rather surprising, as all four $k[V]$-modules
$H^0(\Omega^1_Y(T_i))$ for $1\le i\le 4$ become isomorphic after a finite extension of the ground field.
We conclude that the $\D_k$ modules $\NN_Y(T_i)$ are non-isomorphic for $1\le i\le 4$.
On the other hand, further computation shows that $\NN_Y(T_4)\simeq \NN_Y(T_5)\simeq \NN_Y(T_6)$
as $\D_k$-modules, which is again somewhat surprising as the torsors $T_4$ and $T_5$
are non-isomorphic, while the torsors $T_1, T_3$ {\em are} isomorphic, as are $T_2,T_4$.
Again, all six Dieudonn\'e modules $\NN_Y(T_i)$ become isomorphic after 
a suitable finite extension on $k$.
\end{s-example}

\section {Applications to $J_Y[p]$: The local-local
  part}\label{s:apps-ll}
In this section, we consider $J_Y[p]_{ll}$ and its Dieudonn\'e module
$H^1_{dR}(Y)_{ll}$  Throughout, we assume $k=\kbar$.
We view the local-local part of $H^1_{dR}(X)$ as
given, and we exploit the $G$-module structure on $H^1_{dR}(Y)_{ll}$ to
find restrictions on its structure as a $\D_k$-module.  We begin by
recording the basic properties of $H^1_{dR}(Y)_{ll}$.

\begin{prop}\label{prop:M-props}
  Write $M$ for $H^1_{dR}(Y)_{ll}$ and let $h=g_X-f_X$.  We have
  \[M[\delta]\cong M/\delta\cong H^1_{dR}(X)_{ll}.\]
  Moreover, $M$ has the following properties:
\begin{enumerate} 
\item $M$ is a free $k[G]$-module of rank $2h$.
\item $M$ is a self-dual, local-local $BT_1$ module, i.e.,
  $\im F=\Ker V$, $\im V=\Ker F$, $F$ and $V$ act nilpotently on $M$,
  and $M$ admits a perfect $k$-bilinear pairing $\<\cdot,\cdot\>$
  which is alternating \textup{(}$\< m,m\>=0$ for all
  $m\in M$\textup{)} and which satisfies $\< Fm,n\>=\< m,Vn\>^p$ for
  all $m,n\in M$.
\item The pairing is compatible with the $G$ action in that
  $\< \gamma m,\gamma n\>=\< m,n\>$ for all $m,n\in M$.
\item $\im F=\Ker V$ and $\im V=\Ker F$ are free submodules of $M$ of rank $h$.
\end{enumerate}
\end{prop}

\begin{proof}
  Proposition~\ref{prop:H^1(Y)} shows that $M[\delta]\cong
  M/\delta\cong  H^1_{dR}(X)_{ll}$.
  Part~(1) was proven in Proposition~\ref{prop:H^1(Y)} part (2).  For
  part~(2), as we reviewed at the beginning of Section~\ref{s:H1dR},
  $H^1_{dR}(Y)$ is a self-dual $BT_1$ module where the duality is
  induced by the de Rham pairing, and it is easy to see that the
  restriction of this pairing makes $M$ into a self-dual $BT_1$
  module.  It is local-local by definition.  Part~(3) holds because
  $\gamma$ is an automorphism of $Y$, so has degree 1.  For part~(4), note
  that $VM=H^0(Y,\Omega^1_Y)_{ll}$, and comparing the result of Tamagawa
  (equation~\eqref{eq:Tamagawa}) to that of Nakajima
  (equation~\eqref{eq:Nakajima}) shows that $VM$ is free over $k[G]$ of
  rank $h$.  The same then follows for
  $\im F=\Ker V$ since $VM\cong M/(\Ker V)$.
\end{proof}

\begin{rem}
  For a module $M$ with properties (1--3), the spaces $\im F=\Ker V$
  and $\im V=\Ker F$ are all free over $k[G]$ of rank $h$ as soon as
  one of them is.  Furthermore, in this situation the calculations
  \[\< Fm,Fn\>=\< m,VFn\>^p=0
    \and  \< Vm,Vn\>=\< m,FVn\>^{1/p}=0\]
 show that $\im F$ and $\im V$ are isotropic, and they are maximal
 isotropic since they have dimension $ph=\frac12\dim M$.
\end{rem}

\subsection{Coordinates}
We will introduce special coordinates on any $\D_k[G]$-module $M$ with
properties (1--4) as in Proposition~\ref{prop:M-props}.  This allows
for numerical experimentation and leads to full analyses in certain
significant cases.

Write $R$ for the group ring $k[G]$ and introduce a $k$-linear
involution $a\mapsto \tilde a$ by requiring that $\tilde g= g^{-1}$
for all $g\in G$.
This involution is trivial if $p=2$ and is non-trivial with invariant
subspace of dimension $(p+1)/2$ if $p>2$.  We extend it to vectors and
matrices with entries in $R$ by acting component-wise.

\begin{lemma}\label{lemma:tilde-inv}
If $p>2$ and $a\in R$ is a non-zero element with $\tilde a=a$, then
$a=\delta^\alpha u$ where $\alpha\in\{0,2,\dots,p-1\}$ \textup{(}i.e., $\alpha$
is even\textup{)} and $u\in R^\times$ \textup{(}i.e., $u$ is a unit\textup{)}.
\end{lemma}

\begin{proof}
  Since $\delta=\gamma-1$, we have
  \[\tilde\delta=\gamma^{-1}-1=-\gamma^{-1}\delta=-\delta/(1+\delta).\]
Writing $a=\delta^\alpha u$ where $\alpha\ge0$ and $u\in R^\times$, 
we see that $\tilde a\equiv(-1)^\alpha a\pmod{\delta^{\alpha+1}R}$
so $\tilde a =a $ implies that $\alpha$ is even.
\end{proof}

Recall that $\gamma\in G$ is the element corresponding to 1 under
$G\cong\Z/p\Z$. 
For $a=a_0+a_1\gamma+\cdots+a_{p-1}\gamma^{p-1}\in R$, define
\[(a)_0 := a_0.\]
Next, note that the function $R\times R\to k$ given by
\[ (a,b):=(a\tilde b)_0\] is $k$-bilinear and satisfies
$(\gamma a,\gamma b)=(a,b)$.  Let $J$ be the $2h$-by-$2h$ matrix
\[J=\begin{pmatrix}0_h&I_h\\-I_h&0_h\end{pmatrix},\]
  where $0_h$ and $I_h$ are the $h\times h$ zero and identity matrices
  respectively.    Regarding $R^{2h}$  as a space of column vectors,
  we have a perfect, alternating, $k$-bilinear pairing on $R^{2h}$ given by
  \[ \< m,n\> = \left(\t m J\tilde n\right)_0.\]
  Here, $\t m$ stands for the transpose of $m$ and $\tilde n$ is
  computed by applying the involution $a\mapsto\tilde a$ to each entry
  of $n$.  If $G$ acts coordinate-wise on elements of $R^{2h}$, then
  $\< \gamma m,\gamma n\>=\< m,n\>$.

  If $M$ is a $\D_k[G]$-module which is free over $R$ of rank $2h$,
  and if $m_1,\dots,m_{2h}$ is an ordered basis of $M$, we write $[m]$
  for the coordinate vector of $m$:
  \[[m]=\begin{pmatrix}r_1\\\vdots\\r_{2h}\end{pmatrix}
    \quad\text{if}\quad m=r_1m_1+\cdots+r_{2h}m_{2h}.\]
  Since $F$ acts $R$-semilinearly (i.e., $F(am)=a^{(p)}Fm$ for
  $a\in R$ and $m\in M$), we may represent $F$ by a $2h$-by-$2h$
  matrix $\FF$ with entries in $R$, namely, the matrix such that
\[\left[Fm\right]=\FF[m]^{(p)}\]
where the right hand side is the matrix product of $\FF$ and
$[m]^{(p)}$.

We say that a $\D_k$-module $N$ of dimension $2h$ over $k$ is
\emph{superspecial} if it is isomorphic to the Dieudonn\'e module of 
$E[p]^h$, with $E$ a
supersingular elliptic curve.  Three equivalent characterizations are:
(i) $F^2=V^2=0$ on $N$; (ii) in the Kraft--Oort classification, $N$
corresponds to the word $fv$ repeated $h$ times; and (iii) In the
Ekedahl--Oort classification, $N$ corresponds to the elementary
sequence $[0,0,\dots,0]$.
  
\begin{prop}\label{prop:coords}
  Suppose that $M$ is a $\D_k[G]$-module with properties
  \textup{(1--4)} as in Proposition~\ref{prop:M-props}.
  \begin{enumerate}
  \item There exists an ordered basis $m_1,\dots,m_{2h}$ of $M$ such
    that the pairing $\<\cdot,\cdot\>$ is given by
\[\< m,n\>=\left(\t[m]J\widetilde{[n]}\right)_0,\]
and such that the matrix of $F$ has the form
\[\FF=\begin{pmatrix}0&B\\0&D\end{pmatrix}\]
where $B$ and $D$ are $h$-by-$h$ matrices with entries in $R$
satisfying $\t\tilde DB=\t\tilde BD$ and the columns of $\FF$ generate
a free $R$-module of rank $h$.
\item Conversely, any choice of $B$ and $D$ satisfying
  $\t\tilde DB=\t\tilde BD$ and such that the columns of
  $\FF=\begin{pmatrix}0&B\\0&D\end{pmatrix}$ generate a free
  $R$-module of rank $h$ and $\FF$ is $p$-nilpotent\footnote{We say
    $\FF$ is ``$p$-nilpotent'' if $\FF\FF^{(p)}\cdots\FF^{(p^a)}=0$
    for some $a>0$.} defines the structure of $\D_k[G]$-module on
  $R^{2h}$ which satisfies properties \textup{(}1--4\textup{)} of
  Proposition~\ref{prop:M-props}.
\item If $M/\delta M$ is superspecial, then we may choose the basis so
  that the matrix of $F$ has the form
\[\FF=\begin{pmatrix}0&I\\0&D\end{pmatrix}\]
where $\delta$ divides $D$ \textup{(}i.e., $\delta$ divides every
entry of $D$\textup{) }and $\t\tilde D=D$.
\item Conversely, any choice of $D$ which is divisible by $\delta$ and
  which satisfies $\t\tilde D=D$ defines the structure of
  $\D_k[G]$-module on $R^{2h}$ which satisfies properties
  \textup{(}1--4\textup{)} of Proposition~\ref{prop:M-props} and which
  is superspecial modulo $\delta$.
  \end{enumerate}
\end{prop}

\begin{rems}\mbox{}
  \begin{enumerate}
  \item In parts (1) and (3), we do not claim that $B$ and $D$ are
    uniquely determined by $M$. 
    \item In parts (2) and (4), the action of $V$ is determined by that
      of $F$ and the pairing by the requirement that $\<Fm,n\>=\< m,Vn\>^p$.
  \end{enumerate}
\end{rems}

\begin{proof}
  By hypothesis, $\Ker F$ is free of rank $h$ over $R$ and isotropic
  for the pairing.  Choose an $R$-basis $m_1,\dots,m_h$ for $\Ker F$.
  Since the pairing is non-degenerate, we may choose elements
  $n_1,\dots,n_h$ of $M$ such that for $0\le i<p$ and $1\le j,\ell\le
  h$ we have
  \[\<\gamma^im_j,n_\ell\>=\begin{cases}1&\text{if $i=0$ and $j=\ell$,}\\
      0&\text{otherwise.}
    \end{cases}\]
  Since
  $\<\gamma^im_j,\gamma^{i'}n_\ell\>=\<\gamma^{i-i'}m_j,n_\ell\>$, we
  find that $n_1,\dots,n_h$ generate a free $R$-module of rank $h$
  which is complementary to $\Ker F$ and in $k$-duality with $\Ker F$
  via the pairing.  We then inductively modify the $n_j$ by elements
  of $\Ker F$ to make their $R$-span isotropic.  More precisely, set
  $m_{h+1}=n_1$, choose $m_{12}\in \Ker F$ such that
  $\<m_{h+1},m_{12}\>=\<m_{h+1},m_{12}\>$ and set
  $m_{h+2}=n_2-m_{12}$, etc.  Then for $1\le i,j\le h$ we have
\[\<m_i,m_j\>=\<m_{h+i},m_{h+j}\>=0\and
  \<m_i,m_{h+j}\>=\begin{cases}1&\text{if $i=j$,}\\
    0&\text{otherwise},
  \end{cases}\]
The pairing $\<\cdot,\cdot\>$ then has the desired form with respect
to the basis $m_1,\dots,m_{2h}$.

Since the first $h$ basis elements span the kernel of $F$, the matrix
of $F$ has the form 
\[\FF=\begin{pmatrix}0&B\\0&D\end{pmatrix}\]
where $B$ and $D$ are $h$-by-$h$ matrices with coordinates in $R${}.
Since $\im F$ is $R$-free of rank $h$, the columns of $\FF$ generate a
free $R$-module of rank $h$.  Let $\VV$ be the matrix of $V$ with
respect to the chosen basis.  The compatibility $\<Fm,n\>=\<m,Vn\>^p$
implies that $\t\FF J=J\tilde\VV^{(p)}$, so
\[\VV=\begin{pmatrix}\t\tilde D^{(1/p)}&-\t\tilde B^{(1/p)}\\
    0&0\end{pmatrix},\]
and $VF=0$ implies $\VV\FF^{(1/p)}=0$ which in turn implies $\t\tilde
DB=\t\tilde BD$.  This completes the proof of part~(1).

For part~(2), given $B$ and $D$ satisfying the conditions, define a
$p$-linear operator $F$ and a $p^{-1}$-linear operator $\VV$ on
$R^{2h}$ by setting 
\[Fm=\FF m^{(p)}\and Vm=\VV m^{(1/p)}\]
where
\[\FF=\begin{pmatrix}0&B\\0&D\end{pmatrix}\and
  \VV=\begin{pmatrix}\t\tilde D^{(1/p)}&-\t\tilde B^{(1/p)}\\
    0&0\end{pmatrix}.\]
Then $\im F=\Ker V$ and $\im V=\Ker F$, so we obtain a $BT_1$-module
with a perfect alternating pairing, and it is straightforward to check
that it has the properties (1--4) enumerated in
Proposition~\ref{prop:M-props}.  This completes the proof of part~(2)

For part~(3), first choose a basis as in part~(1).  Since $M/\delta M$
is superspecial, the matrix $D$ is divisible by $\delta$, and the
condition that the columns of $\FF$ span a free $R$-module of rank $h$
implies that $B$ is invertible.

Now consider changes of coordinates that preserve the matrix of
the pairing.  These are precisely the matrices $S\in GL_{2h}(R)$
satisfying $\t SJ\tilde S=J$.  In particular, we may take $S$ of the
form 
\[S=\begin{pmatrix}T&0\\0&U\end{pmatrix}\]
where $T,U\in\GL_h(R)$ and $\t T\tilde U=I$.  
in terms of the new basis, the matrix of $F$ is
\[S^{-1}\FF S^{(p)}=\begin{pmatrix}0&T^{-1}BU^{(p)}\\
    0&U^{-1}DU^{(p)}\end{pmatrix}
=\begin{pmatrix}0&\t\tilde U BU^{(p)}\\
    0&U^{-1}DU^{(p)}\end{pmatrix}.\]
By the Lang--Steinberg theorem \cite[Thm.~10.1]{SteinbergELAG} (applied
to the endomorphism $U\mapsto
\left(\t\tilde U^{-1}\right)^{(p)}$ of $\GL_h(R)$), we may choose
$U$ so that $\t\tilde U BU^{(p)}=I$.  In these new
coordinates, $\FF$ has the desired form, and this proves part~(3).

Part~(4) follows from the same argument as in part~(2) and the observation
that $\FF$ is $p$-nilpotent as soon as $\delta$ divides $D$.
\end{proof}

\begin{rem}
The coordinates in the theorem provide another way to see that there
can be infinitely many non-isomorphic $\D_k[G]$-modules with the same
underlying $\D_k$-module and $k[G]$-module.  In brief, one considers the
subgroup of $\GL_{2h}(R)$ preserving the form and the subspace $\Ker F$.
Assuming that $M/\delta M$ is superspecial and passing to the subgroup
preserving the shape of $\FF$ in part (3) (i.e., with an identity
matrix in the upper right), we find matrices of block form $\begin{pmatrix}
  T&U\\0&W \end{pmatrix}$ where (among other conditions)
${}^t\tilde WW^{(p)}$ is congruent to the identity modulo $\delta$.
Such a matrix $W$ is congruent modulo $\delta$ to one with entries in
$\F_{p^2}$. On the other hand, the effect of this change of
coordinates on $D$ is $D\mapsto W^{-1}DW^{(p)}$, so the claim follows
if one knows that the $\D_k$-module structure of $M$ is determined by
fairly crude invariants of $D$ such as its rank and the rank of
(Frobenius-twisted) iterates. We will not write down a general result
of this shape, but the proofs of Theorem~\ref{thm:h=1} and part (1) of
Theorem~\ref{thm:p=2} give examples where this is visible.
\end{rem}

The coordinates of Proposition~\ref{prop:coords} can be used to
analyze the  special case where $h=1$. Recall that the
$a$-number of a $\D_k$-module is $a=\dim_k\left(\Ker F\cap\Ker
  V\right)$. 

\begin{thm}\label{thm:h=1}
  Assume $p>2$ and let $M$ be a $\D_k[G]$-module satisfying the
  properties \textup{(1--4)} of Proposition~\ref{prop:M-props} with
  $h=1$.
  The $a$-number of $M$ is in $\{2,4,\dots,p-1,p\}$.  Define integers
  $\ell$ and $b$ by $0\le b<a$ and $p=\ell a+b$.  Then
  \[M\cong M\left(f^{\ell+1}v^{\ell+1}\right)^{b}\bigoplus
    M\left(f^{\ell}v^{\ell}\right)^{a-b}\]
  In other words, $M$ can be presented as the $\D_k$-module with
  generators $e_1,\dots,e_a$ and relations
  \[F^{\ell+1}e_i=V^{\ell+1}e_i\quad\text{for $1\le i\le b$}\and
  F^{\ell}e_i=V^{\ell}e_i\quad\text{for $b< i\le a$}.\]
\end{thm}

We give other descriptions of $M$ in the proof below.

\begin{proof}
  Since $M/\delta M$ is 2-dimensional over $k$ and local-local, it is
  superspecial.  We apply part~(3) of Proposition~\ref{prop:coords} to
  identify $M$ with $R^2$ where $F$ acts via a matrix of the form
  $\begin{pmatrix}0&1\\0&D\end{pmatrix}$ with $D\in R$ satisfying
  $\tilde D=D$, and $V$ acts via
  $\begin{pmatrix}D^{(1/p)}&-1\\0&0\end{pmatrix}$.

   We will use the Kraft--Oort and Ekedahl--Oort classifications of
  $\D_k$-modules to analyze $M$.  (See \cite{Oort01} for the original
  analysis and \cite{PriesUlmer21} for a more leisurely and detailed
  exposition.)  First, we construct the canonical filtration on $R^2$,
  and then we check that under the Ekedahl--Oort classification, $M$
  has elementary sequence $[0,\dots,0,1,2,\dots,p-a]$ (of length $p$
  with $a$ zeroes), or equivalently under the Kraft--Oort
  classification, it corresponds to the words $f^{\ell+1}v^{\ell+1}$ with
  multiplicity $b$ and $f^\ell V^\ell$ with multiplicity $a-b$.
  
  We first treat the edge case $D=0$. If $D=0$,
  then one easily checks that the canonical filtration on $M$ has the
  form
  \[0=M_0\subset M_1=R\begin{pmatrix}1\\0\end{pmatrix}\subset M_2=M,\]
  with $FM=M_1$ and $FM_1=0$.  Thus $M$ has elementary sequence
  $[0,\dots,0]$, and $a$ number $p$.  In the Kraft--Oort
  classification, this corresponds to the word $fv$ with multiplicity
  $p$.

  Now assume that $D\neq 0$.  Then $D$ has the form $\delta^au$ where
  $0<a<p$ and $u\in R^\times$.  Moreover, by
  Lemma~\ref{lemma:tilde-inv}, $a$ must be even.  Writing $p=\ell a+b$
  with $b<a$, we have that $b\neq0$.  Define submodules
  $M_j\subset M=R^2$ for $0\le j\le 4\ell+2$ by
  \begin{align*}
    M_{2i}&=R\delta^{p-ia}\begin{pmatrix}1\\D\end{pmatrix}
         &\text{for $0\le i\le \ell$}\\
    M_{1+2i}&=R\delta^{p-b-ia}\begin{pmatrix}1\\D\end{pmatrix}
         &\text{for $0\le i\le \ell$}\\
    M_{2\ell+1+2i}&=R\begin{pmatrix}1\\D\end{pmatrix}
                               +R\delta^{(p-ia)}\begin{pmatrix}0\\1\end{pmatrix}
         &\text{for $0\le i\le \ell$}\\
    M_{2\ell+2+2i}&=R\begin{pmatrix}1\\D\end{pmatrix}
                               +R\delta^{(p-b-ia)}\begin{pmatrix}0\\1\end{pmatrix}
         &\text{for $0\le i\le \ell$}\\
  \end{align*}
  Then we have
  \[0=M_0\subset M_1\subset\cdots\subset M_{2\ell+1}=FM\subset
    M_{2\ell+2}\subset\cdots\subset M_{4\ell+2}=M.\]

  Next, one checks that
  \[FM_j=\begin{cases}
      M_0&\text{if $0\le j\le 2$,}\\
      M_{j-2}&\text{if $2\le j\le 2\ell+1$,}\\
      M_{2\ell-1}&\text{if $2\ell+1\le j\le 4\ell$,}\\
      M_{2\ell}&\text{if $j=4\ell+1$,}\\
      M_{2\ell+1}&\text{if $j=4\ell+2$,}
    \end{cases}\]
  and
  \[V^{-1}M_j=\begin{cases}
      M_{2\ell+1}&\text{if $j=0$,}\\
      M_{2\ell+2}&\text{if $j=1$,}\\
      M_{2\ell+3}&\text{if $2\le j\le 2\ell+1$,}\\
      M_{j+2}&\text{if $2k+1\le j\le 4\ell$,}\\
      M_{4\ell+2}&\text{if $4k\le j\le 4\ell+2$.}
    \end{cases}\]
  This shows that the $M_j$ give the canonical filtration of $M$, and
  that the corresponding elementary sequence is
  $[0,\dots,0,1,\dots,p-a]$.  Thus the $a$-number of $M$ is $a$.  For
  the classification by words, we note that the cycles corresponding
  to the filtration are
  \[0\labeledlongto{V^{-1}}2\ell+1\labeledlongto{V^{-1}}2\ell+3
    \labeledlongto{V^{-1}}\cdots 
    \labeledlongto{V^{-1}}4\ell+1\labeledlongto{F}2\ell\labeledlongto{F}2\ell-2
    \labeledlongto{F}\cdots
   \labeledlongto{F}2\labeledlongto{F}0\]
 (yielding the word $f^{\ell+1}v^{\ell+1}$) with multiplicity $b=\dim_k(M_1/M_0)$ and
 \[1\labeledlongto{V^{-1}}2\ell+2\labeledlongto{V^{-1}}2\ell+4
   \labeledlongto{V^{-1}}\cdots
   \labeledlongto{V^{-1}}4\ell\labeledlongto{F}2\ell-1
   \labeledlongto{F}2\ell-3\labeledlongto{F}\cdots
   \labeledlongto{F}3\labeledlongto{F}1\]
 (yielding the word $f^{\ell}v^{\ell}$)  with multiplicity $a-b=\dim_k(M_2/M_1)$.

 The presentation by generators and relations then follows from
 \cite[Lemma~3.1]{PriesUlmer21}.  This completes the proof of the
 theorem.
  \end{proof}

We can now give a significant application to Artin--Schreier covers:

\begin{cor}\label{cor:h=1}
  Suppose that $p>2$, $\pi:Y\to X$ and $G=\gal(Y/X)\cong\Z/p\Z$ are as usual
  and that $f_X=g_X-1$.
  Then the $a$-number of $J_Y$ is in
  $\{2,4,\dots,p-1,p\}$.  Moreover, the Dieudonn\'e module of $J_Y[p]$
  has the form
  \[  L\oplus (L\tensor k[G])^{f_X-1}\oplus M_a\]
  where $L=M(\Z/p\Z\oplus\mu_p)$ and $M_a$ is the module described in
  Theorem~\ref{thm:h=1}.
\end{cor}

\begin{proof}
  Since $k$ is algebraically closed, $H^1_{dR}(Y)_\et$ is completely
  split, and part~(2) of Proposition~\ref{prop:H^1(Y)} gives its
  $G$-module structure.  In all we have an isomorphism of
  $\D_k[G]$-modules 
  \[H^1_{dR}(Y)_\et\cong M(\Z/p\Z)\oplus\left(M(\Z/p\Z)\tensor_k
      k[G]\right)^{f_X-1}.\]
  Similarly,
  \[H^1_{dR}(Y)_m\cong M(\mu_p)\oplus\left(M(\mu_p)\tensor_k
      k[G]\right)^{f_X-1}.\]
  Since $H^1_{dR}(Y)_{ll}=M$ has the properties enumerated in
  Proposition~\ref{prop:M-props}, it is isomorphic to the
  $\D_k[G]$-module $M_a$ described in Theorem~\ref{thm:h=1}.  This
  completes the proof of the Corollary.
\end{proof}

\begin{s-example}
Let $p=5$ and $k=\F_p$.  The table below illustrates that all
possibilities listed in Corollary \ref{cor:h=1} 
occur, with $X$ hyperelliptic of degree $7$ and genus $g_X=3$.
Each base curve $X$ has $f_X=2=g_X-1$, and $\nu_X=1$ so the specified
cover $Y\rightarrow X$ 
is the unique unramified $\Z/p\Z$-cover defined over $k$.
\begin{center}
\begin{tabular}{c|c|c|}
	$X: y^2 =  $ &  $Y: z^3-z = $ & $a_Y$ \\
	\hline 
	$x^7 - x^5 - 2x^3 - 2x^2 + x - 1$  & $(-2x^4 - x^2 + 1)y + 2$ & 2 \\
	$-x^7 - x^6 - x^5 + x^4 + x^2 - 2x$  &  $(-x^4 - 2x^3 + 2x^2 + 1)y$ & 4 \\
	$2x^7 - 2x^5 + 2x^4 - x$  & $(-x^9 + 2x^7 - 2x^6 - x^5 + 2x^4 - x + 1)y$ & 5 \\
\end{tabular}
\end{center}
\end{s-example}

The argument underlying Theorem~\ref{thm:h=1} has consequences for
other $\D_k[G]$-modules $M$ with $M/\delta M$ superspecial:

\begin{thm}\label{thm:more-a-parity}
  Suppose that $p>2$ and let $M$ be a $\D_k[G]$-module with properties
  \textup{(1--4)} as in Proposition~\ref{prop:M-props} \textup{(}in
  particular, free over $R$ of rank $2h$\textup{)} and such that
  $M/\delta M$ is superspecial.  Choose coordinates as in
  Proposition~\ref{prop:coords}, and let $a(M)$ be the $a$-number of
  $M$.
  \begin{enumerate}
  \item If $D$ has the form $\delta^\alpha U$ with $1\le\alpha\le p$
    and $U\in\GL_h(R),$ then $M$ is the direct sum of $h$ copies of
    the module in Theorem~\ref{thm:h=1}.  I.e., writing
    $p=\ell\alpha+\beta$ with $0\le\beta<\alpha$, $M$ is the module
    associated to the words $f^{\ell+1}v^{\ell+1}$ with multiplicity
    $h\beta$ and $f^\ell v^\ell$ with multiplicity $h(\alpha-\beta)$.
    If $\alpha<p$, then $h\alpha$ is even.
  \item If $a(M)=h$ \textup{(}the minimum possible value\textup{)},
    then $M$ is associated to the word $f^pv^p$ with multiplicity $h$.
    In this case, $h$ is even.
  \item If $h$ is odd, then $a(M)\ge h+1$.
  \item If $a(M)<p$, then $a(M)$ is even.      
  \end{enumerate}
\end{thm}

\begin{proof}
  We begin by observing that if
  \[D=\sum_{i=\alpha}^{p-1}\delta^iD_i\]
  with the $D_i\in M_h(k)$, then an argument generalizing that of
  Lemma~\ref{lemma:tilde-inv} shows that the ``leading matrix''
  $D_\alpha$ is symmetric when $\alpha$ is even and skew-symmetric
  when $\alpha$ is odd.  We also note that
  \[\Ker F\cap\Ker V=\Ker F\left|_{\im F}\right.,\]
  so
  \begin{equation}\label{eq:a-D}
a(M)=\dim_k\Ker\left(D:R^h\to R^h\right).    
  \end{equation}

(1) If $D$ is a power of $\delta$ times an invertible matrix, the
argument of Theorem~\ref{thm:h=1} calculating the canonical filtration
goes through essentially verbatim with
$\begin{pmatrix}1\\D\end{pmatrix}$ replaced by the column span of
$\begin{pmatrix}I\\D\end{pmatrix}$ and
$\begin{pmatrix}0\\1\end{pmatrix}$ replaced by
$\begin{pmatrix}0\\R^h\end{pmatrix}$.  The canonical filtration of $M$ has the
same length and same permutation as in Theorem~\ref{thm:h=1}, but the
subquotients are $h$ times as large. For the parity assertion, the
remarks above show that if $h\alpha$ is odd, then the leading matrix
$D_\alpha$ is skew-symmetric and of odd size, so not invertible.

(2) The assumption that $M/\delta M$ is superspecial implies that
$M[\delta]$ is also superspecial, so we have $a(M)\ge h$.  By \eqref{eq:a-D},
we can have equality only if $\delta^2$ does not divide $D$ and the
leading matrix $D_1$ is invertible.  Thus we are in the case
$\alpha=1$ of part (1), $\beta=0$, and $\ell=p$, so $M$ is associated
to $f^pv^p$ with multiplicity $h$, and $h$ must be even.

(3) If $h$ is odd, we cannot be in case (2), so $a(M)>h$.

(4) Equation~\eqref{eq:a-D} shows that as an element of $R$,
$\det D\in\delta^{a(M)}R^\times$, in other words, $a(M)$, the length
of $\Ker D$, is the ``valuation'' of $\det D$.  (This is literally
true only if we lift $D$ to a valuation ring, say $k[[\delta]]$.
Projecting back to $R=k[[\delta]]/(\delta^p)$ gives the statement
here.)  Since ${}^t\tilde D=D$, we have $\widetilde{\det D}=\det D$,
so if $a(M)<p$, then $a(M)$ must be even by
Lemma~\ref{lemma:tilde-inv}.
\end{proof}

Theorem~\ref{thm:h=1} and the first half of
Theorem~\ref{thm:more-a-parity} show that if $h$ or $a$ is small, the
possibilities for $M$ are severely restricted.  We can also make such a
deduction when $a$ is large with respect to $h$:

\begin{thm}\label{thm:large-a}
  Let $M$ be a $\D_k[G]$-module satisfying the
  properties \textup{(1--4)} of Proposition~\ref{prop:M-props}.
  \begin{enumerate}
  \item If $a(M)=ph$ \textup{(}the maximum value\textup{)}, then $M$
    and $M/\delta M$ are superspecial.
  \item If $a(M)=ph-1$, then $M/\delta M$ is superspecial and there
    are exactly $h$ possibilities for 
    $M$ as a $\D_k$-module which are indexed by an integer $1\le r\le h$.
    The corresponding Ekedahl--Oort structure is
    $[0,\dots,0,1,\dots,1]$ with $ph-r$ zeroes and $r$ ones.  The
    corresponding multiset of words is $fv$ with multiplicity $2h-2r$
    and the word $(vf)^{r-1}v^2(fv)^{r-1}f^2$ with multiplicity one.
  \end{enumerate}
\end{thm}

Note that the Ekedahl--Oort structures appearing in part (2) are a
small fraction ($h$ vs. ${ph}$) of those with $a$-number 1, i.e., the
$R$-structure on $M$ is a significant constraint.

\begin{proof}
  (1) By Equation~\eqref{eq:a-D}, $M$ has $a$-number $ph$ precisely
  when $D=0$, and it follows (as in the proof of
  Theorem~\ref{thm:h=1}) that $M$ and $M/\delta M$ are superspecial.

  (2) By Equation~\eqref{eq:a-D}, $M$ has $a$-number $ph-1$ precisely
  when $D$ has $k$-rank 1, in which case we must have
  $D=\delta^{p-1}D_{p-1}$ where $D_{p-1}\in M_h(k)$ has rank 1.  It
  follows that $D\equiv0\pmod\delta$, i.e., $M/\delta M$ is
  superspecial.

  Choose a vector $v\in k^h$ so that the column span of
  $D_{p-1}$ is $kv$, and let $r$ be defined by
  \[r:=\dim_k\left( kv+kv^{(p^2)}+kv^{(p^4)}+\cdots\right).\]
  (Here we write $v^{(p^a)}$ for the result of raising each coordinate
  of $v$ to the $p^a$-th power.)  Note that $1\le r\le h$.  We are
  going to show that $M$ has the asserted invariants by computing
  its canonical filtration.

  We use the definitions and results of \cite{Oort01} as reviewed in
  \cite[\S3-4]{PriesUlmer21}. We index certain elements of the
  canonical filtration of $M$ by their dimensions.  In particular,
  \[M_{ph}=\Ker V=\im F=\text{column span (over $R$) of }\begin{pmatrix}I\\D
    \end{pmatrix}.\]
  Let $\phi:\{0,\dots,2ph\}\to\{0,\dots,ph\}$ be the function such
  that $FM_i=M_{\phi(i)}$.  (To define $\phi$ on this domain, we need
  to extend the canonical filtration to a final filtration, but that
  will play no role in what follows.)  To say that the $a$-number of
  $M$ is 1 is precisely to say that $\phi(ph)=1$.  Let
    \[M_1=F^2M=k\begin{pmatrix}\delta^{p-1}v^{(p)}\\0
      \end{pmatrix}\]
    and compute 
    \begin{align*}
      M_{ph+1}&=V^{-1}M_1=M_{ph}+
           k\begin{pmatrix}0\\\delta^{p-1}v^{(p^2)}\end{pmatrix}\\
      M_2&=FM_{ph+1}=
           k\begin{pmatrix}\delta^{p-1}v^{(p)}\\0\end{pmatrix}+
      k\begin{pmatrix}\delta^{p-1}v^{(p^3)}\\0\end{pmatrix}\\
      M_{ph+2}&=V^{-1}M_2=M_{ph}+
           k\begin{pmatrix}0\\\delta^{p-1}v^{(p^2)}\end{pmatrix}+
           k\begin{pmatrix}0\\\delta^{p-1}v^{(p^4)}\end{pmatrix}\\
      &\vdots\\
      M_{ph+r}&=V^{-1}M_2=M_{ph}+
           k\begin{pmatrix}0\\\delta^{p-1}v^{(p^2)}\end{pmatrix}+
      k\begin{pmatrix}0\\\delta^{p-1}v^{(p^4)}\end{pmatrix}+
      \cdots+k\begin{pmatrix}0\\\delta^{p-1}v^{(p^{2r})}\end{pmatrix}
    \end{align*}
Note that the definition of $r$ implies that $FM_{ph+r}$ has dimension
$r$.   We have thus established that $\phi(ph+i)=i+1$ for $1\le i<r$
and $\phi(ph+r)=r$.  The self-duality of $M$ implies that
$\phi(ph-i)=1$ for $1\le i<r$ and $\phi(ph-r)=0$.  This completely
determines the Ekedahl --Oort structure of $M$ as
$[0,\dots,0,1,\dots,1]$ with $ph-r$ zeroes and $r$ ones.

We leave it as an exercise for the reader to check that the
corresponding set of words is $fv$ with multiplicity $2h-2r$
and $(vf)^{r-1}v^2(fv)^{r-1}f^2$ with multiplicity one.
\end{proof}

We now turn to the case with $p=2$ and $M/\delta M$ superspecial where
we are able to give an almost complete analysis.  Before stating the
main result, we record a simple lemma.

\begin{lemma}\label{lemma:p=2-R-structure}
  Suppose $p=2$ and let $M$ be a finite-dimensional, local-local
  $\D_k$-module with a fixed alternating pairing $\<\cdot,\cdot\>$
  satisfying $\<Fm,n\>=\<m,Vn\>^p$ for all $m,n\in M$.  In order to
  equip $M$ with the structure of a $\D_k[G]$-module with the
  properties \textup{(1-4)} of Proposition~\ref{prop:M-props}, it
  suffices to give a $\D_k$-submodule $N\subset M$ which is maximal
  isotropic for the pairing and an isomorphism $\phi:N^\vee\to N$ of
  $\D_k$-modules which is symmetric, i.e., with $\phi^\vee=\phi$.
\end{lemma}

\begin{proof}
Since $N$ is maximal isotropic, the pairing on $M$ induces
an isomorphism of $\D_k$-modules $M/N\cong N^\vee$.  We define
$\delta:M\to M$ as the composition 
\[M\onto M/N\cong N^\vee\labeledto{\phi}N\into M.\]
Clearly $\delta^2=0$, so $M$ has the structure of an $R=k[G]$-module.
If $m_1,\dots,m_d$ is a basis of a complement of $N$ in $M$, then the
same set is a basis of $M$ as an $R$-module, so $M$ is $R$-free.

Since each of the maps in the composition defining $\delta$ is a map
of $\D_k$-modules, so is $\delta$, so $M$ is a $\D_k$-module
satisfying parts (1-2) of Proposition~\ref{prop:M-props}.

For part (3), the pairing and the $R$-action are compatible if and
only if $\<\delta m_1,m_2\>=\<m_1,\delta m_2\>$ for all
$m_1,m_2\in M$. (Here $p=2$, so $\tilde\delta=\delta$.) This holds
because we chose $\phi$ to be symmetric.

For part (4), since $M$ and $N$ are self-dual,
\[\dim_kF(M)={1\over 2}\dim_k(M)\and
  \dim_kF(N)={1\over2}\dim_k(N)={1\over4}\dim_k(M).\]
Choose a basis $m_1,\dots,m_{d/2}$ for a complement of $F(N)$ in
$F(M)$. Then these elements span an $R$-module of $k$-dimension $d$.
Since $\delta$ commutes with $\D_k$, it is contained in $F(M)$, and
thus equal to $F(M)$ by dimension considerations. This shows that
$\im F$ is a free $R$-module and completes the proof that $M$ is
a $\D_k$-module with properties (1-4) of
Proposition~\ref{prop:M-props}. \end{proof}

\begin{thm}\label{thm:p=2}
  Assume that $p=2$, and define a $\D_k[G]$-module $M$ to be
  ``admissible'' if $M$ has the properties \textup{(1--4)} of
  Proposition~\ref{prop:M-props} and $M/\delta M$ is superspecial.
  \begin{enumerate}
  \item The multiset of words attached to an admissible $M$ consist of
    $fv$ with even multiplicity together with a self-dual multiset of
    words each of which is a concatenation of subwords of the form
    \[(vf)^{e_2}v^2(fv)^{e_1}f^2\]
    with $e_1,e_2\ge0$.
  \item Every multiset of words $\WW$ as in part \textup{(1)} is
    the set of words attached to an admissible $M$.
  \item The Ekedahl--Oort structure attached to an admissible $M$
    starts with $h$ zeroes, i.e., it has the form
    \[ [0,\dots,0,\psi_{h+1},\dots,\psi_{2h}].\]
  \item Every Ekedahl--Oort structure
    $\Psi=[0,\dots,0,\psi_{h+1},\dots,\psi_{2h}]$ of length $2h$ which
    starts with $h$ zeroes is the Ekedahl--Oort structure of an admissible $M$.
  \end{enumerate}
\end{thm}

\begin{rem}
  We will prove part (2) of the theorem with an explicit
  construction of a module $M$ with the asserted structures as $\D_k$-
  and $R=k[G]$-modules. We do not claim that any admissible $M$ is
  isomorphic as $\D_k[G]$-module to one of those constructed in part
  (2). (This is the reason we describe our analysis as
  ``almost complete.'')  In fact, the examples constructed for part
  (2) have a direct factor (as $\D_k[G]$-module) whose underlying
  $\D_k$-module is  $M(fv)$ with even multiplicity.  However,
  Example~\ref{ex:indecomp} after the proof shows that not every
  admissible module has this property.
\end{rem}

\begin{proof}
(1) Let $M$ be an admissible $\D_k[G]$-module, and choose coordinates
as in Proposition~\ref{prop:coords}.  We will use these coordinates to
compute certain elements of the canonical filtration of $M$.

Since $M/\delta M$ is superspecial, the matrix $D$ in
Proposition~\ref{prop:coords}(3) has the form $D=\delta D_1$ where
$D_1\in M_h(k)$.  For vectors or matrices with coordinates in $k$ or
subspaces of $k^h$, we use an exponent ${}^{(p^a)}$ to denote the
result of raising each entry or coordinate to the power $p^a$.
Viewing $D_1$ as an endomorphism of $k^h$ by left multiplication,
define subspaces 
\begin{align*}
  K&=\Ker D_1\cap \Ker D_1^{(p^2)}\cap \Ker D_1^{(p^4)}\cap\cdots\\
\noalign{and}  
  I&=\im D_1+\im D_1^{(p^2)}+\im D_1^{(p^4)}+\cdots%\\
\end{align*}
Since $k^h$ is finite-dimensional, each of these subspaces is in fact
the intersection or sum of finitely many of the displayed terms, and
we have $K^{(p^2)}=K$ and $I^{(p^2)}=I$.

We let words on the alphabet $\{f,v\}$ act on submodules of $M$ as in
\cite[\S3]{PriesUlmer21}: $fN=F(N)$ and $vN=V^{-1}(N)$.
Straightforward calculation (similar to that in the proof of
Theorem~\ref{thm:large-a}(2)) shows that for sufficiently large $n$,
$(vf)^nM$ is independent of $n$.  We write $(vf)^\infty M$ for the
common value, and we find
\[(vf)^\infty M=
  \begin{pmatrix}
    R^h\\ \delta R^h+K^{(p)}
  \end{pmatrix}.\]
Similarly,
\begin{align*}
  (vf)^\infty 0 &=  \begin{pmatrix}
    R^h\\ \delta I
  \end{pmatrix},\\
  (fv)^\infty M &=  \begin{pmatrix}
    \delta R^h+ K\\ 0
  \end{pmatrix},\\
  \noalign{and}
    (fv)^\infty 0 &=  \begin{pmatrix}
    \delta I^{(p)}\\ 0
  \end{pmatrix}.
\end{align*}
Note that $f$ and $v$ carry the quotients $\left((vf)^\infty
  M\right)/\left((vf)^\infty 0\right)$  and $\left((fv)^\infty
  M\right)/\left((fv)^\infty 0\right)$ isomorphically onto one
another.  It follows that the multiplicity of $fv$ in the multiset of
words attached to $M$ is
\[\dim_k K+\codim_k I = 2 \dim_k K.\]
Set $e=\dim_k K$.

Choose a complement $J\subset k^h$ to $I$ (i.e., $k^h=I\oplus J$) with
$J^{(p^2)}=J$.   Then we find that the subspace
\[N:=\begin{pmatrix}
    0\\ K^{(p)}+\delta J
  \end{pmatrix}+
  \begin{pmatrix}
    \delta K+\delta J^{(p)}\\ 0
  \end{pmatrix}\]
of $M$ is a $\D_k$-submodule isomorphic to $M(fv)^{2e}$.  (Note that it
is not in general an $R$-submodule.)  Moreover, it is visible that the
image $N$ in $M/\delta M$ is isomorphic to $e$ copies of $M(fv)$.

By the Kraft classification of $\D_k$-modules, we have an isomorphism
$M\cong N\oplus N'$ of $\D_k$-modules where $N'$ has dimension $4h-4e$
over $k$ and where the multiset of words associated to $N'$ is
self-dual and consists of primitive words of length $>2$ (i.e., the
words $f$, $v$, and $fv$ do not appear).  Since the image of $N$ in
$M/\delta M$ has dimension $2e$ and is isomorphic to $M(fv)^e$, the
image of $N'$ in $M/\delta M$ must have dimension at least $2h-2e$ and
contain a subspace isomorphic to $M(fv)^{h-e}$.  In particular, the
$u_{1,1}$-number of $N'$ is at least $h-e$.

We now check that for $p=2$, a self-dual $\D_k$-module $N'$ of
dimension $4n$ and $u_{1,1}$-number $\ge n$ whose associated set of
primitive words consists of words of length $>2$ in fact has
$u_{1,1}$-number equal to $n$ and all of its associated words are
concatenations of words of the form $(vf)^{e_2}v^2(fv)^{e_1}f^2$. To
that end, recall the calculation of $u_{1,1}$-numbers in
\cite{PriesUlmer21correction}: Given a word $w$, rotate and subdivide
it so that it has the form $w=s_1s_2\cdots s_\ell$ where each $s_i$
ends in $f^2$ and has no other appearances of $f^2$. A word that ends
in $f^2$, has no other appearances of $f^2$, and contains $v^2$ can be
written in the form $(vf)^{e_2}v^2t_i(fv)^{e_1}f^2$ where $t_i$ is
either empty or it starts with $f$ and ends with $v^2$. (This defines
$e_1$ and $e_2$.) By \cite[Prop.]{PriesUlmer21correction},
\[u(M(w))+u(M(w)^\vee)=\sum_{i=1}^\ell u_i\] where
\[u_i= \begin{cases}
  0&\text{if $s_i$ has no appearances of $v^2$,}\\
 (e_2+1)+(e_1+1)&\text{if $s_i=(vf)^{e_2}v^2t_i(fv)^{e_1}f^2$.}
\end{cases}\]
We find that
\[u(M(w))+u(M(w)^\vee) \le \frac14\dim_k\left(M(w)\oplus
    M(w)^\vee\right)\]
with equality if and only if the second case above always holds and
all of the $t_i$ are empty. This establishes that the words associated
to $N'$ are concatentations of words of the form
$(vf)^{e_2}v^2(fv)^{e_1}f^2$, which completes the proof of part (1) of
the theorem.

(2) Suppose $fv$ appears in $\WW$ with even multiplicity $d$.
Give $N:=M(fv)\tensor_kR$ the structure of a $\D_k[G]$-module by
letting $\D_k$ act on the first factor and $R$ on the second. Choose
an alternating pairing on $M(fv)$ and use it to identify
$M(fv)\tensor\delta$ with the dual of $M(fv)\tensor 1$.
Lemma~\ref{lemma:p=2-R-structure} shows that $N$ is admissible, as is
$N^{d/2}$. To finish the proof, we construct an admissible module with
multiset of words equal to those in $\WW$ except $fv$. Taking
the direct sum with $N^{d/2}$ then provides the required module.

We thus may assume that $\WW$ does not contain $fv$, so consists
entirely of words which are concatenations of words of the form
$(vf)^{e_2}v^2(fv)^{e_1}f^2$ with $e_1,e_2\ge0$. Let $M$ be the
$\D_k$-module associated to the words in $\WW$ and let $4h=\dim_kM$,
so $4h$ is also the number of letters in $\WW$. We claim that there is
a surjection of $\D_k$-modules $M\onto M(fv)^h$ whose kernel is
isomorphic as $\D_k$-module to $M(fv)^h$. Choosing a symmetric pairing
on $M(fv)^h$ and applying Lemma~\ref{lemma:p=2-R-structure}, we see
that $M$ admits a $\D_k[G]$-module structure which makes it
admissible.

The proof of \cite[Proposition]{PriesUlmer21correction} shows that
each subword $(vf)^{e_2}v^2(fv)^{e_1}f^2$ leads to a (non-unique)
surjection $M\onto M(fv)^{e_2+1}$ and (dually) a (non-unique) injection
$M(fv)^{e_1+1}\into M$. With appropriate choices, these fit together to
make $M$ into an extension of $M(fv)^h$ by itself. We explain this for
a single subword and leave the mild generalization to concatenations
to the reader.

Let $w=(vf)^{e_2}v^2(fv)^{e_1}f^2$.   Present the module $M(w)$ as in
\cite[\S3.2]{PriesUlmer21} with generators
$E_0,E_1,\dots,E_{e_1+e_2+1}=E_0$ and relations
\begin{align*}
  VE_{e_1+e_2+1}&=FE_{e_1+e_2},\\
                &\vdots\\
  V^2E_{e_1+1}&=FE_{e_1},\\
                &\vdots\\
  VE_1&=F^2E_0.
\end{align*}
(For notational convenience we have assumed $e_1,e_2>0$.  The other
cases are similar but simpler.)  We write $z_0^{(j)}$ and $z_1^{(j)}$
(with $j=1,\dots,e_2+1$) for bases of copies of $M(fv)$ where 
$Fz_0^{(j)}=Vz_0^{(j)}=z_1^{(j)}$ and $Fz_1^{(j)}=Vz_1^{(j)}=0$.  A
surjection to $M(fv)^{e_2+1}$ is given by 
\[E_i\mapsto
  \begin{cases}
    0&\text{$1\le i\le e_1$,}\\
    \alpha_j^{p^{2(i-e_1-1)}}z_0^{(i-e_1)}&\text{$e_1+1\le i\le e_1+e_2+1$,}
  \end{cases}\]
where $\alpha_1,\dots,\alpha_{e_2+1}\in k$ are chosen so that the
matrix
\[
  \begin{pmatrix}
    \alpha_1&\alpha_1^{p^2}&\dots&\alpha_1^{p^{2e_2}}\\
    \vdots&\vdots&&\vdots\\
    \alpha_{e_2+1}&\alpha_{e_2+1}^{p^2}&\dots&\alpha_{e_2+1}^{p^{2e_2}}
  \end{pmatrix}\]
has rank $e_2+1$.  Let $\beta_0,\beta_1,\dots,\beta_{e_2+1}$ be a
non-zero solution to the system
\[
  \begin{pmatrix}
    \alpha_1^{1/p}&\alpha_1^p&\alpha_1^{p^3}&\dots&\alpha_1^{p^{2e_2+1}}\\
    \vdots&\vdots&\vdots&&\vdots\\
    \alpha_{e_2+1}^{1/p}&\alpha_{e_2+1}^p&\alpha_{e_2+1}^{p^3}
    &\dots&\alpha_{e_2+1}^{p^{2e_2+1}}
  \end{pmatrix}
  \begin{pmatrix}
    \beta_0\\ \vdots \\ \beta_{e_2+1}
  \end{pmatrix}=0,
\]
and note that by our choice of the $\alpha_i$, we have that $\beta_0$ and
$\beta_{e_2+1}$ are non-zero.  One then checks that the kernel of the
surjection $M\onto M(fv)^h$ is the $k$-vector space $N$ spanned by
\[
  \beta_0VE_{e_1+1}+\beta_1FE_{e_1+1}+\cdots+\beta_{e_2+1}FE_{e_1+e_2+1}\]
and
\[E_1,\dots, E_{e_1},
  F^2E_0, FE_1,\dots FE_{e_1}.\]
Let $N_1$ be the subspace spanned by the first $e_1+1$ vectors listed
above.  The one finds that $FN=VN=N_1$ and $FN_1=VN_1=0$.  This shows
that $N$ is isomorphic as $\D_k$-module to $M(fv)^{e_1+1}$.

Making a similar analysis of concatenations of words of the form
$(vf)^{e_2}v^2(fv)^{e_1}f^2$ shows that $M$ is an extension of
$M(fv)^h$ by $M(fv)^h$, and this completes the proof of part (2) of
the Theorem.

For part (3), assume that $M$ is associated to the a set of cyclic words
consisting of $fv$ with multiplicity $2d$ and concatenations of words
\[s_i=(vf)^{e_{2i}}v^2(fv)^{e_{2i-1}}f^2\qquad i=1,\dots,m\]
(with repetitions as necessary to account for multiplicities).  Note
that self-duality implies that
$\sum_{i=1}^me_{2i}=\sum_{i=1}^me_{2i-1}$.  Let
$h=d+\sum_{i=1}^me_{2i}$ so that $M$ has dimension $4h$ over $k$.

A basis of $M$ is indexed by the words obtained by lifting each cyclic
word above to a word in all possible rotations. In the standard model
of $M(w)$, a basis vector associated to a word ending in $f$ is mapped
by $F$ to the basis vector corresponding to the first rotation of the
word, and a basis vector corresponding to a word ending with $v$ is
killed by $F$. (See \cite[3.2.1]{PriesUlmer21} for more details.)

We recall from \cite[\S3.5]{PriesUlmer21} a total order on words on
$\{f,v\}$: given two words $w_1$ and $w_2$, form powers $w_1^a$ and
$w_2^b$ of the same length, and then compare them in lexicographic
order (from the left with $f<v$). The canonical filtration of $M$ is
obtained by taking subspaces generated by the vectors associated to
words less than or equal to one of the words which appears (with
repetitions for multiplicities).

We now consider which words in $\WW$ are less than or equal to $fv$.
They are those which start with $f^2\dots$ or with $(fv)^jf^2\dots$
with $1\le j \le e_{2i-1}$ and $1\le i\le m$ as well as $fv$ itself
(taken with multiplicities). Note that all of these words end in $v$.
The total number of them is $2d+\sum_{i=1}^me_{2i-1}=h+d$. Since words
that end in $v$ are killed by $F$, we see that the subspace $M_{d+h}$
in the canonical filtration of $M$ with dimension $d+h$ is killed by
$F$, and this means that the Ekedahl--Oort structure associated to $M$
begins with $h+d$ zeroes.  This establishes the claim in part (3) of
the theorem.

  (4) We consider the $\D_k$-module $M$ with E--O structure $\Psi$ as
  constructed by Oort in \cite[\S9]{Oort01} and show that $M$ admits a
  $k[G]$-module structure which makes it admissible.

  Extend $\Psi$ to a ``final sequence'' $[\psi_1,\dots,\psi_{4h}]$ by
  setting $\psi_{4h}=2h$ and $\psi_{4h-i}=\psi_i+2h-i$ for $1\le i\le
  2h$.  Let $1\le m_1<m_2<\cdots<m_{2h}\le 4h$ be the indices $i$ such
  that $\psi_{i-1}<\psi_i$ and let $1\le n_{2h}<n_{2h-1}<\cdots
  n_1\le 4h$ be the indices $i$ such that $\psi_{i-1}=\psi_i$.  Our
  hypothesis on $\Psi$ implies that $m_{h+i}=3h+i$ for $1\le i\le h$
  and $n_{2h+1-i}=i$ for $1\le i\le h$.  Moreover, $m_i+n_i=4h+1$ for
  $1\le i\le 2h$.

  Now let $M$ be the $k$-vector space with basis $Z_1,\dots,Z_{4h}$
  and for $1\le i\le 2h$ define $X_i=Z_{m_i}$ and $Y_i=Z_{n_i}$.
  Define a $\D_k$ module structure on $M$ by setting
  \[ F(X_i)=Z_i,\quad F(Y_i)=0,\quad V(Z_i)=0,\and V(Z_{4h+1-i})=Y_i
      \quad\text{for $1\le i\le 2h$}.\]
  Introduce a bilinear pairing $\<\cdot,\cdot\>$ on $M$ by setting
  \[\<X_i,X_j\>=0,\quad\<Y_i,Y_j\>=0,\and\<X_i,Y_j\>=\delta_{ij}
    \quad\text{for $1\le i,j\le 2h$}.\]
  It is then straightforward to check that $M$ is a self-dual, local-local
  $BT_1$ module whose elementary sequence is the given $\Psi$.

  Let $N$ be the subspace of $M$ spanned by $X_i+Y_i$ and $Y_{h+i}$
  for $1\le i\le h$.  Then $N$ is a $\D_k$-submodule and we have
  \[F(N)=V(N)=\text{Span of $Y_{h+i}$ for $1\le i\le h$.}\]
  Since $F(Y_{h+i})=V(Y_{h+i})=0$ for $1\le i\le h$, it follows that
  $N$ is superspecial.

  The quotient $M/N$ is spanned by the classes of $X_j$ for $1\le j\le
  2h$, and we have
  \[F(M/N)=V(M/N)=\text{Span of the classes of $X_i$ for $1\le i\le
      h$.}\]
  Since $F(X_{i})=V(X_{i})=0\pmod N$ for $1\le i\le h$,  it follows
  that $M/N\cong N^\vee$ is also superspecial. Choosing a symmetric
  isomorphism $N^\vee\to N$
  and applying Lemma~\ref{lemma:p=2-R-structure}, we find the desired 
  $R$-module structure on $M$. 

 This completes the proof of the Theorem.
\end{proof}

\begin{s-example}\label{ex:indecomp}
Let $p=2$ and consider the $\D_k[G]$-module $M$ of rank $2h=4$ over $R$ given
as in Proposition~\ref{prop:coords}(4) by $D=\delta D_1$ with
\[D_1=
  \begin{pmatrix}
    1&1\\1&1
  \end{pmatrix}.\]
Note that $D_1^2=0$.  The Ekedahl--Oort structure associated to $M$ is
$[0,0,0,1]$ and the associated set of words is $(f^2v^2),(fv)^2$.  The
submodule $M_7:=V^{-1}FM$ has dimension 7, and one checks that if
$m\in M\setminus M_7$, then the $\D_k[G]$-submodule of $M$ generated by
$m$ has dimension 6 and is isomorphic as $R$-module to $R^2+\delta
R^2$.  In particular, the smallest free $R$-submodule of $M$
containing it is all of $M$.  This proves that $M$ is indecomposable
as a $\D_k[G]$-module.  Note by contrast that the $\D_k$-module
associated to $(f^2v^2),(fv)^2$ in part (2) of Theorem~\ref{thm:p=2}
decomposes into a free $R$-module of rank 2 supporting the
$\D_k$-module $(fv)^2$ and another of rank 2 supporting $(f^2v^2)$.
This shows that not every admissible $\D_k[G]$-module arises from the
construction in part (2).
\end{s-example}

\bibliography{database}

\end{document}

%% file: formatting.tex
\numberwithin{equation}{section}

\theoremstyle{plain}
\newtheorem{thm}[subsection]{Theorem}

\newtheorem{prop}[subsection]{Proposition}

\newtheorem{lemma}[subsection]{Lemma}
\newtheorem{cor}[subsection]{Corollary}

\newtheorem*{thm*}{Theorem}
\newtheorem*{mthm*}{Main Theorem}
\newtheorem{prop-def}[subsection]{Proposition-Definition}
\newtheorem{def-prop}[subsection]{Definition-Proposition}
\newtheorem{def-lemma}[subsection]{Definition-Lemma}

\theoremstyle{definition}

\newtheorem*{defn*}{Definition}

\theoremstyle{remark}
\newtheorem{rem}[subsection]{Remark}
\newtheorem{rems}[subsection]{Remarks}

\newtheorem{s-example}[subsection]{Example}
\newtheorem{s-examples}[subsection]{Examples}

%% file: macros.tex
\usepackage[OT2,T1]{fontenc}
\DeclareSymbolFont{cyrletters}{OT2}{wncyr}{m}{n}
\DeclareMathSymbol{\sha}{\mathalpha}{cyrletters}{"58}

\newcommand{\CC}{\mathcal{C}}

\newcommand{\JJ}{\mathcal{J}}
\newcommand{\KK}{\mathcal{K}}
\newcommand{\QQ}{\mathcal{Q}}

\newcommand{\VV}{\mathcal{V}}
\newcommand{\WW}{\mathcal{W}}

\newcommand{\GG}{{\mathcal{G}}}
\newcommand{\HH}{{\mathcal{H}}}

\renewcommand{\O}{\mathcal{O}}
\newcommand{\FF}{\mathcal{F}}

\newcommand{\MM}{\mathcal{M}}
\newcommand{\NN}{\mathcal{N}}
\newcommand{\OO}{\mathcal{O}}

\newcommand{\F}{\mathbb{F}}
\newcommand{\Fp}{{\mathbb{F}_p}}

\newcommand{\kbar}{{\overline{k}}}

\newcommand{\Z}{\mathbb{Z}}

\newcommand{\R}{\mathbb{R}}

\renewcommand{\P}{\mathbb{P}}

\newcommand{\<}{\langle}
\renewcommand{\>}{\rangle}
\newcommand{\into}{\hookrightarrow}
\newcommand{\onto}{\twoheadrightarrow}
\newcommand{\isoto}{\,\tilde{\to}\,}

\newcommand{\tensor}{\otimes}

\newcommand{\nodiv}{\not|}

\def\nodiv{\mathrel{\mathchoice{\not|}{\not|}{\kern-.2em\not\kern.2em|}
{\kern-.2em\not\kern.2em|}}}

\newcommand{\GL}{\mathrm{GL}}

\newcommand{\labeledto}[1]{\overset{#1}{\to}}
\newcommand{\labeledlongto}[1]{\overset{#1}{\longrightarrow}}

\newcommand{\psmat}[1]{\bigl(\begin{smallmatrix}#1\end{smallmatrix}\bigr)}

\DeclareMathOperator{\codim}{codim}

\DeclareMathOperator{\im}{Im}
\DeclareMathOperator{\Ker}{Ker}
\DeclareMathOperator{\coker}{Coker}

\DeclareMathOperator{\res}{Res}

\DeclareMathOperator{\Hom}{Hom}
\DeclareMathOperator{\aut}{Aut}

\DeclareMathOperator{\gal}{Gal}

\DeclareMathOperator{\Fr}{Fr}

\DeclareMathOperator{\Spec}{Spec}

\def\clap#1{\hbox to 0pt{\hss#1\hss}}

\makeatletter
\newcommand*\bigcdot{\mathpalette\bigcdot@{.5}}
\newcommand*\bigcdot@[2]{\mathbin{\vcenter{
               \hbox{\scalebox{#2}{$\m@th#1\bullet$}}}}}
\makeatother